\documentclass{article}
\usepackage[height=8.6in,width=6.1in]{geometry}
\usepackage[latin1]{inputenc}
\usepackage[sort,numbers]{natbib}
\usepackage{graphicx}
\usepackage{amsmath,amssymb,amsfonts,amsxtra,mathrsfs,bm}
\usepackage{amsthm}
\usepackage{algorithm}
\usepackage{algorithmic}
\usepackage{multicol}
\usepackage{multirow}
\usepackage[neverdecrease]{paralist}
\usepackage{xspace}
\usepackage{microtype}
\usepackage{caption}
\usepackage{subcaption}
\usepackage[colorlinks=true,citecolor=blue]{hyperref}

\newcommand{\ms}{S}
\newcommand{\vx}{x}
\newcommand{\mx}{X}
\newcommand{\md}{D}
\newcommand{\mr}{R}

\newcommand{\Xc}{\mathcal{X}}
\newcommand{\Hc}{\mathcal{H}}
\newcommand{\Lc}{\mathcal{L}}
\newcommand{\Escr}{\mathscr{E}}
\newcommand{\Mc}{\mathcal{M}}
\newcommand{\Gc}{\mathcal{G}}
\newcommand{\Fc}{\mathcal{F}}

\newcommand{\pp}{\mathbb{P}}
\newcommand{\C}{\mathbb{C}}

\newcommand{\inv}[1]{#1^{-1}}

\newcommand{\nlsum}{\sum\nolimits}
\newcommand{\nlmin}{\min\nolimits}
\newcommand{\nlprod}{\prod\nolimits}
\newcommand{\reals}{\mathbb{R}}

\newcommand{\posdef}{\mathbb{P}}
\newcommand{\pd}{\posdef}
\renewcommand{\H}{\mathbb{H}}

\newcommand{\riem}{\delta_R}

\newcommand{\thom}{\delta_T}

\newcommand{\gm}{{\#}}

\newcommand{\da}{\downarrow}

\newcommand{\pfrac}[2]{\left(\tfrac{#1}{#2}\right)}

\newcommand{\frob}[1]{\|{#1}\|_{\text{F}}}

\newcommand{\mynorm}[2]{\| {#1} \|_{#2}}
\newcommand{\norm}[1]{\mynorm{#1}{}}

\newcommand{\half}{\tfrac{1}{2}}

\newcommand{\set}[1]{\left\{ {#1}\right\}}

\newcommand{\kron}{\otimes}

\def\nlpkron{\begingroup\textstyle \bigotimes\nolimits\endgroup}
\def\sint{\begingroup\textstyle \int\endgroup} % small integral

\DeclareMathOperator*{\argmin}{argmin}

\DeclareMathOperator{\trace}{tr}
\DeclareMathOperator{\Diag}{Diag}

\DeclareMathOperator{\rank}{rank}

\newtheorem{theorem}{Theorem}
\newtheorem{lemma}[theorem]{Lemma}
\newtheorem{prop}[theorem]{Proposition}
\newtheorem{corr}[theorem]{Corollary}
\theoremstyle{definition}
\newtheorem{defn}[theorem]{Definition}
\newtheorem{example}[theorem]{Example}
\newtheorem{remark}[theorem]{Remark}

\begin{document}
\title{Conic geometric optimisation on the manifold of positive definite matrices}
\author{Suvrit Sra\thanks{A part of this work was done while the author was at Carnegie Mellon University, Pittsburgh (Machine Learning Dept.) on leave from the MPI for Intelligent Systems, T\"ubingen, Germany. A preliminary fraction of this paper was presented at the \emph{Advances in Neural Information Processing Systems, 2013} conference.}
  \and
  Reshad Hosseini\thanks{School of ECE, College of Engineering, University of Tehran, Tehran, Iran}}
%\slugger{siopt}{xxxx}{xx}{x}{x--x}%slugger should be set to mms, siap, sicomp, sicon, sidma, sima, simax, sinum, siopt, sisc, or sirev

\maketitle

% \begin{center}
%   \large\textsc{Working manuscript}
% \end{center}

\renewcommand{\floatpagefraction}{.95}

\begin{abstract}
  We develop \emph{geometric optimisation} on the manifold of Hermitian positive definite (HPD) matrices. In particular, we consider optimising two types of cost functions: (i) geodesically convex (g-convex); and (ii) log-nonexpansive (LN). G-convex functions are nonconvex in the usual euclidean sense, but convex along the manifold and thus allow global optimisation. LN functions may fail to be even g-convex, but still remain globally optimisable due to their special structure. We develop theoretical tools to recognise and generate g-convex functions as well as cone theoretic fixed-point optimisation algorithms. We illustrate our techniques by applying them to maximum-likelihood parameter estimation for elliptically contoured distributions (a rich class that substantially generalises the multivariate normal distribution). We compare our fixed-point algorithms with sophisticated manifold optimisation methods and obtain notable speedups.
\end{abstract} 

\begin{center}
\end{center}

\pagestyle{myheadings}
\thispagestyle{plain}

\section{Introduction}
Hermitian positive definite (HPD) matrices possess a remarkably rich geometry that is a cornerstone of modern convex optimisation~\citep{nestNem94} and convex geometry~\citep{parrilo}.  In particular, HPD matrices form a convex cone, the strict interior of which is a differentiable Riemannian manifold which is also a prototypical CAT(0) space (i.e., a metric space of nonpositive curvature~\citep{bridson}). This rich structure enables ``geometric optimisation'' on the set of HPD matrices---enabling us to solve certain problems that may be nonconvex in the Euclidean sense but are convex in the manifold sense (see \S\ref{sec.gc} or~\citep{wie12}), or failing that, still have enough geometry (see \S\ref{sec.ln}) so as to admit efficient optimisation.

This paper formally develops \emph{conic geometric optimisation}\footnote{To our knowledge the name ``geometric optimisation'' has not been previously attached to g-convex and cone theoretic HPD matrix optimisation, though several scattered examples do exist. Our theorems offer a formal starting point for recognising HPD geometric optimisation problems.} for HPD matrices. We present key results that help recognise geodesic convexity (g-convexity); we also present sufficient conditions that place even several non geodesically convex functions within the grasp of geometric optimisation. %To our knowledge, ours are the most general results on geometric optimisation with HPD matrices known so far.

\subsection*{Motivation} We begin by noting that the widely studied class of \emph{geometric programs} ultimately reduces to conic geometric optimisation on $1\times 1$ HPD matrices (i.e., positive scalars; see Remark~\ref{rmk.logc}). Geometric programming has enjoyed great success across a spectrum of applications---see e.g., the survey of~\citet{boyd.gp}; we hope this paper helps conic geometric optimisation gain wider exposure.

Perhaps the best known conic geometric optimisation problem is computation of the Karcher (Fr\'echet) mean of a set of HPD matrices, a topic that has attracted great attention within matrix theory~\citep{bhatia11,ssdiv,bruno,jeuVaVa}, computer vision~\citep{vemuri}, radar imaging~\citep[Part II]{nieBha13}, medical imaging~\citep{dti,vemuri12}---we refer the reader to the recent book~\citep{nieBha13} for additional applications and references. Another basic geometric optimisation problem arises as a subroutine in image search and matrix clustering~\citep{chSra12}. 

Conic geometric optimisation problems also occur in several other areas: statistics (covariance shrinkage)~\citep{chen11}, nonlinear matrix equations~\citep{leeLim}, Markov decision processes and more broadly in the fascinating areas of nonlinear Perron-Frobenius theory~\citep{lemNuss12}. 

As a concrete illustration of our ideas, we discuss the task of \emph{maximum likelihood estimate} (mle) for \emph{elliptically contoured distributions} (ECDs)~\citep{cambanis81,gupta99,muirhead82}---see \S\ref{sec.ecd}. We use ECDs to illustrate our theory, not only because of their instructive value but also because of their importance in a variety of applications~\citep{ollila11}. 

\subsection*{Outline} The main focus of this paper is on recognising and constructing certain structured nonconvex functions of HPD matrices. In particular, Section~\ref{sec.gc} studies the class of geodesically convex functions, while Section~\ref{sec.ln} introduces ``log-nonexpansive'' functions. We present a limited-memory BFGS algorithm in Section~\ref{sec.manopt}, where we also present a derivation for the \emph{parallel transport}, which, we could not find elsewhere in the literature. Even though manifold optimisation algorithms apply to both classes of functions, for log-nonexpansive functions we advance fixed-point theory and algorithms separately in Section~\ref{sec.ln}.  We present an application of geometric optimisation in Section~\ref{sec.ecd}, where we consider statistical inference with elliptically contoured distributions. Numerical results are the subject of Section~\ref{sec.expt}.

\section{Geodesic convexity for HPD matrices}
\label{sec.gc}
Geodesic convexity (g-convexity) is a classical concept in geometry and analysis; it is used extensively in the study of Hadamard manifolds and metric spaces of \emph{nonpositive curvature}~\citep{bridson,athan}, i.e., metric spaces having a g-convex distance function. The concept of g-convexity has been previously explored in nonlinear optimisation~\citep{rapsak}, but its  importance and applicability in statistical applications and optimisation has only recently gained more attention~\citep{wie12,zhang13}. It is worth noting that geometric programming~\citep{boyd.gp} ultimately relies on ``geometric-mean'' convexity~\citep{niculescu}, i.e., $f(x^\alpha y^{1-\alpha}) \le [f(x)]^\alpha[f(y)]^{1-\alpha}$, which is nothing but logarithmic g-convexity on $1\times 1$ HPD matrices (positive scalars).

To introduce g-convexity on $n\times n$ HPD matrices we begin by recalling some key definitions---see \citep{bridson,athan} for extensive details. 

\begin{defn}[g-convex sets]
  Let $\Mc$ be a $d$-dimensional connected $C^2$ Riemannian manifold. A set $\Xc \subset \Mc$ is called \emph{geodesically convex} if any two points of $\Xc$ are joined by a geodesic lying in $\Xc$. That is, if $x, y \in \Xc$, then there exists a shortest path $\gamma: [0,1] \to \Xc$ such that $\gamma(0)=x$ and $\gamma(1)=y$.
\end{defn}

\begin{defn}[g-convex functions]
  Let $\Xc \subset \Mc$ be a g-convex set. A function $\phi : \Xc \to \reals$ is called \emph{geodesically convex}, if for any $x, y \in \Xc$, we have the inequality
  \begin{equation}
    \label{eq.2}
   \phi(\gamma(t)) \le (1-t)\phi(\gamma(0)) + t\phi(\gamma(1)) = (1-t)\phi(x) + t\phi(y),
  \end{equation}
  where $\gamma(\cdot)$ is the geodesic $\gamma: [0,1] \to \Xc$ with $\gamma(0)=x$ and $\gamma(1)=y$.
\end{defn}

\subsection{Recognising g-convexity}
Unlike scalar g-convexity, for matrices recognising g-convexity is not so easy. Indeed, for scalars, a function $f: \reals_{++}\to \reals$ is log-g-convex (and hence g-convex) if and only if $\log \circ f \circ \exp$ is convex. A similar characterisation does not seem to exist for HPD matrices, primarily due to the noncommutativity of matrix multiplication. We develop some theory below for helping recognise and construct g-convex functions.

To define g-convex functions on HPD matrices recall that $\pp_d$ is a differentiable Riemannian manifold where geodesics between points are available in closed form. Indeed, the tangent space to $\pp_d$ at any point can be identified with the set of Hermitian matrices, and the inner product on this space leads to a Riemannian metric on $\pp_d$. At any point $A \in \pp_d$, this metric is given by the differential form $ds = \frob{A^{-1/2}dAA^{-1/2}}$; for $A, B \in \pp_d$ there is a unique geodesic path~\citep[Thm.~6.1.6]{bhatia07} %\commR{Probably it is unique for a certain inner product definition on the tangent space. \citet{wie12}[p. 32] mentioned that the metric is logarithmic metric. \citet{zhang12}[p. 3] mentioned that the metric which leads to the unique geodesic is given by: 
% \begin{equation}
% \label{eq.metr}
% || \log (A^{-\tfrac12} B  A^{-\tfrac12})  ||_{F}
% \end{equation}
% }
\begin{equation}
  \label{eq.14}
  \gamma(t) = A \gm_t B := A^{1/2}(A^{-1/2}BA^{-1/2})^tA^{1/2},\quad t \in [0,1].
\end{equation}
The midpoint of this path, namely $A\gm_{1/2}B$ is called the \emph{matrix geometric mean}, which is an object of great interest~\citep{bhatia07,bhatia11,jeuVaVa,bruno}---we drop the $1/2$ and denote it simply by $A\gm B$. Starting from the geodesic~\eqref{eq.14}, many g-convex functions can be constructed by extending monotonic convex functions to matrices. To that end, first recall the fundamental operator inequality~\citep{ando79} (where $\preceq$ denotes the L\"owner partial order):
\begin{equation}
  \label{eq:1}
  A\gm_t B \preceq (1-t)A+ tB.
\end{equation}
Theorem~\ref{thm.trace.gc} uses the operator inequality~\eqref{eq:1} to construct ``tracial'' g-convex functions.

\begin{theorem}
  \label{thm.trace.gc}
  Let $h : \reals_+ \to \reals$ be monotonically increasing and convex; let $\lambda : \pp_n \to \reals_+^n$ denote the eigenvalue map and $\lambda^\da(\cdot)$ its decreasingly sorted version.  Then, $\nlsum_{j=1}^k h(\lambda_j^\da(\cdot))$ is g-convex for each $1\le k \le n$. 
\end{theorem}
\begin{proof}
  It suffices to establish midpoint convexity. Inequality~\eqref{eq:1} implies that
  \begin{equation*}
    %\label{eq:2}
    \lambda_j(A\gm B) \le \lambda_j\pfrac{A+B}{2},\quad \text{for}\quad 1\le j \le n. % \prec \half\lambda^\downarrow(A) + \half \lambda^\downarrow(B), %\quad\implies \lambda(A\gm B) \prec_w \half\lambda^\downarrow(A) + \half \lambda^\downarrow(B),
  \end{equation*}
  Since $h$ is monotonic, for $1\le k \le n$ it follows that
  \begin{equation}
    \label{eq:3}
    \nlsum_{j=1}^kh(\lambda_j^\da(A\gm B)) \le \nlsum_{j=1}^kh(\lambda_j^\da\pfrac{A+B}{2}).
  \end{equation}
  Lidskii's theorem~\citep[Thm.III.4.1]{bhatia97} yields the majorisation  $\lambda^\da\pfrac{A+B}{2} \prec \tfrac{\lambda^\downarrow(A)+\lambda^\downarrow(B)}{2}$, which combined with a celebrated result of~\citet{haLiPo}\footnote{For a more recent textbook exposition, see e.g., \citep[Theorem~1.5.4]{niculescu}.} and convexity of $h$ yields
  \begin{equation*}
    \sum_{j=1}^kh(\lambda_j^\da\pfrac{A+B}{2}) \le \sum_{j=1}^k h\bigl(\tfrac{\lambda_j^\da(A)+\lambda_j^\da(B)}{2}\bigr) \le \half\sum_{j=1}^k h(\lambda_j^\da(A))+\half \sum_{j=1}^k h(\lambda_j^\da(B)).
  \end{equation*}
  Now invoke inequality~\eqref{eq:3} to conclude that $\nlsum_{j=1}^k h(\lambda_j^\da(\cdot))$ is g-convex.
\end{proof}

\begin{example}
  Theorem~\ref{thm.trace.gc} shows that the following functions are g-convex: (i) $\phi(A)=\trace(e^A)$; (ii) $\phi(A)=\trace(A^\alpha)$ for $\alpha \ge 1$; (iii) $\lambda_1^\da(e^A)$; (iv) $\lambda_1^\da(A^\alpha)$ for $\alpha \ge 1$.
\end{example}

\noindent We now construct examples of g-convex functions different from those obtained via Theorem~\ref{thm.trace.gc}. Let us start with a motivating example.
\begin{example}
  \label{eg.one}
  Let $z \in \C^d$. The function $\phi(A) := z^*A^{-1}z$ is g-convex.
  To prove this claim it suffices to verify midpoint convexity: $
    \phi(A\gm B) \le \half\phi(A) + \half\phi(B)$ for $A, B \in \pp_d$.
  Since $(A\gm B)^{-1} = A^{-1}\gm B^{-1}$ and $\inv{A} \gm \inv{B} \preceq \tfrac{\inv{A}+\inv{B}}{2}$ (\citep[4.16]{bhatia07}), it follows that $\phi(A\gm B) = z^*(A\gm B)^{-1}z \le \half(z^*A^{-1}z + z^*\inv{B}z)=\half(\phi(A)+\phi(B))$.
\end{example}

Below we substantially generalise this example; but first some background.

\begin{defn}[Positive linear map]
  A linear map $\Phi$ from Hilbert space $\Hc_1$ to a Hilbert space $\Hc_2$ is called \emph{positive}, if for $0 \preceq A \in \Hc_1$, $\Phi(A) \succeq 0$. It is called \emph{strictly positive} if $\Phi(A) \succ 0$ for $A \succ 0$; finally, it is called \emph{unital} if $\Phi(I) = I$.
\end{defn}

%The following basic result on positive linear maps will soon prove to be very useful.
\begin{lemma}[{\cite[Ex.4.1.5]{bhatia07}}]
  \label{lem.parsum}
  Define the \emph{parallel sum} of HPD matrices $A, B$ as
  \begin{equation*}
    \label{eq.20}
    A : B := [\inv{A} + \inv{B}]^{-1}.
  \end{equation*}
  Then, for any positive linear map $\Pi: \pp_d \to \pp_k$, we have
  \begin{equation*}
    \Phi(A:B) \preceq \Phi(A): \Phi(B).
  \end{equation*}
\end{lemma}

Building on Lemma~\ref{lem.parsum}, we are ready to state a key theorem that helps us recognise and construct g-convex functions (see Thm.~\ref{thm.gc}, for instance). This result is by itself not new---e.g., it follows from the classic paper of~\citet{kuboAndo80}; due to its key importance we provide our own proof below for completeness.
\begin{theorem}
  \label{thm.linmap}
  Let $\Phi: \pp_d \to \pp_k$ be a strictly positive linear map. Then, 
  \begin{equation}
    \label{eq.8}
    \Phi(A \gm_t B) \preceq \Phi(A) \gm_t \Phi(B),\qquad  t \in [0,1],\ \text{for } A, B \in \pp_d.
  \end{equation}
\end{theorem}
\begin{proof}
  The key insight of the proof is to use the integral identity~\citep{andoLiMa04}:
  \begin{equation*}
    \int_0^1\frac{\lambda^{\alpha-1}(1-\lambda)^{\beta-1}}{[\lambda a^{-1}+(1-\lambda)b^{-1}]^{\alpha+\beta}} d\lambda = \frac{\Gamma(\alpha)\Gamma(\beta)}{\Gamma(\alpha+\beta)}a^\alpha b^\beta
  \end{equation*}
  Using $\alpha=1-t$ and $\beta=t>0$, for $C\succeq 0$ this yields the integral representation 
  \begin{equation}
    \label{eq.18}
    C^t = \frac{\Gamma(1)}{\Gamma(t)\Gamma(1-t)}\int_0^1 \frac{\bigl[\lambda C^{-1} + (1-\lambda)I \bigr]^{-1}}{\lambda^t(1-\lambda)^{1-t}}d\lambda,
  \end{equation}
  where $\Gamma$ is the usual Gamma function.
  Since $A\gm_t B = A^{1/2}(A^{-1/2}BA^{-1/2})^tA^{1/2}$, using~\eqref{eq.18} we may write it as
  \begin{equation}
    \label{eq.19}
    A\gm_tB = \sint_0^1 \bigl[(1-\lambda)A^{-1}+\lambda B^{-1}\bigr]^{-1}d\mu(\lambda),
  \end{equation}
  for a suitable measure $d\mu(\lambda)$. Applying the map $\Phi$ to both sides of~\eqref{eq.19} we obtain
  \begin{align*}
    \Phi(A\gm_tB) &= \sint_0^1\Phi\bigl(\bigl[(1-\lambda)A^{-1}+\lambda B^{-1}\bigr]^{-1}\bigr)d\mu(\lambda)\\
    &= \sint_0^1 \Phi( \bar{A} : \bar{B})d\mu(\lambda),
  \end{align*}
  where $\bar{A} = (1-\lambda)^{-1}A$ and $\bar{B}=\lambda^{-1}B$. Using Lemma~\ref{lem.parsum} and linearity of $\Phi$ we see
  \begin{equation*}
    \begin{split}
      \sint_0^1 \Phi( \bar{A} : \bar{B})d\mu(\lambda) &\preceq\ \ \sint_0^1 \left(\Phi(\bar{A}) : \Phi(\bar{B})\right)d\mu(\lambda)\\
      &=\ \ \sint_0^1\left[(1-\lambda)\Phi(A)^{-1} + \lambda\Phi(B)^{-1}\right]^{-1}d\mu(\lambda)\\
      &\stackrel{\eqref{eq.19}}{=}\ \Phi(A) \gm_t \Phi(B),
    \end{split}
  \end{equation*}
  which completes the proof.
  %From Choi's inequality~\citep{choi} we know that $\Phi(A)^{-1} \le 
\end{proof}

A corollary of Theorem~\ref{thm.linmap} (that subsumes Example~\ref{eg.one}) follows.
\begin{corr}
  \label{corollary.trace}
  Let $A, B \in \pp_d$, and let $X \in \C^{d \times k}$ have full column rank; then
  \begin{equation}
    \label{eq.13}
    \trace X^*(A\gm_t B)X \le [\trace X^*AX]^{1-t}[\trace X^*BX]^{t},\qquad t \in [0,1].
  \end{equation}
\end{corr}
\begin{proof}
  Use the positive linear map $A \mapsto \trace X^*AX$ in Theorem~\ref{thm.linmap}.
\end{proof}
\begin{remark}
  \label{rmk.logc}
  Corollary~\ref{corollary.trace} actually proves a result  stronger than g-convexity: it shows \emph{log-g-convexity}, i.e., $\phi(X\gm Y) \le \sqrt{\phi(X)\phi(Y)}$, so that $\log\phi$ is
  g-convex. It is easy to verify that if $\phi_1,\phi_2$ are log-g-convex, then both $\phi_1\phi_2$  and $\phi_1+\phi_2$ are log-g-convex. % (for proving the latter, use the inequality $\sqrt{ab}+\sqrt{cd} \le \sqrt{(a+c)(b+d)}$ for positive $a,b,c,d$).
\end{remark}

\begin{remark}
  \label{rmk.logc2}
  More generally, if $h: \reals_+\to \reals_+$ is nondecreasing and log-convex, then the map $A\mapsto \nlsum_{i=1}^k\log h(\lambda_i(A))$ is g-convex. The proof is the same as of  Theorem~\ref{thm.trace.gc}. For instance, if $h(x) = e^x$, we obtain the special case that $A\mapsto \log\trace(e^A)$ is g-convex, i.e.,
% , then $\phi(A)=\trace(e^A)$ is log-g-convex, i.e., $\log\trace(e^A)$ is g-convex.
  % Since $g(x)=e^x$ is multiplicatively convex, using Thm.~\ref{thm.gc.sval} we see that $\log\nlprod_{i=1}^ne^{\lambda_i(A)}=\nlsum_{i=1}^n\log\log\trace(e^A)$ is g-convex.
  \begin{equation*}
    \log\nlsum_{i=1}^n e^{\lambda_i(A\gm B)} \le \log\nlsum_{i=1}^ne^{\frac{\lambda_i(A)+\lambda_i(B)}{2}} \le \half\log\nlsum_{i=1}^ne^{\lambda_i(A)}+\half\log\nlsum_{i=1}^ne^{\lambda_i(B)}.
  \end{equation*}
  %where the latter inequality follows from convexity of $\log(\nlsum_i\exp(\cdot))$.
\end{remark}

We mention now another corollary to Theorem~\ref{thm.linmap}; we note in passing that it subsumes a more complicated result of Gurvits and Samorodnitsky~\citep[Lem.~3.2]{gurvits}.
\begin{corr}
  \label{corollary.det}
  Let $A_i \in \C^{d\times k}$ with $k \le d$ such that $\rank([A_i]_{i=1}^m)=k$; also let $B \succeq 0$. Then $\phi(X) := \log\det(B+\nlsum_i A_i^*XA_i)$ is g-convex on $\pp_d$.
\end{corr}
\begin{proof}
  By our assumption on $A_i$ and $B$, the map $\Phi = \ms \mapsto B+\nlsum_iA_i^*XA_i$ is strictly positive. Thm.~\ref{thm.linmap} implies that $\Phi(X\gm Y) = B+\nlsum_i A_i^*(X \gm Y)A_i \preceq \Phi(X)\gm\Phi(Y)$. This operator inequality is stronger than what we require. Indeed, since $\log\det$ is monotonic and determinants are multiplicative, from this inequality it follows that
  \begin{equation*}
    \begin{split}
      \phi(\ms\gm\mr) &=\log\det\Phi(\ms\gm\mr) \le \log\det(\Phi(\ms) \gm \Phi(R))\\
      &\le \half\log\det\Phi(S) + \half\log\det\Phi(R) = \half\phi(\ms)+\half\phi(\mr).
      \end{split}
  \end{equation*}
Observe that the above result extends to $\phi(X) = \log\det\bigl(B+\int_0^\infty A_{\lambda}^*XA_{\lambda}d\mu(\lambda)\bigr)$, where $\mu$ is some positive measure on $(0,\infty)$.
\end{proof}

\begin{remark}
  \label{rmk.det}
Corollary~\ref{corollary.det} may come as a surprise to some readers because $\log\det(X)$ is well-known to be concave (in the Euclidean sense), and yet $\log\det(B+A^*XA)$ turns out to be g-convex---moreover, $\log\det(X)$ is g-linear, i.e., both g-convex and g-concave.
\end{remark}

\begin{example}
  In~\citep{ssdiv} (see also \cite{chSra12,chebbi}) a dissimilarity function   to compare a pair of HPD matrices is studied. Specifically, for $X, Y \succ 0$, this function is called the S-Divergence and is defined as
  \begin{equation}
    \label{eq:18}
    S(X,Y) := \log\det\pfrac{X+Y}{2} - \half\log\det(X)-\half\log\det(Y).
  \end{equation}
  Divergence~\eqref{eq:18} proves useful in several applications~\citep{ssdiv,chSra12,chebbi}, and very recently its joint g-convexity (in both variables) was discovered~\citep{ssdiv}. Corollary~\ref{corollary.det} along with Remark~\ref{rmk.det} yield g-convexity of $S(X,Y)$ in either $X$ or $Y$ separately.
\end{example}

We are now ready to state our next key g-convexity result. A similar result was obtained in~\citep{zhang13}; our result is somewhat more general as it allows incorporation of positive linear maps. Moreover, our proof technique is completely different.
\begin{theorem}
  \label{thm.gc}
  Let $h: \pp_k \to \reals$ be nondecreasing (in L\"owner order) and g-convex. Let $r\in\set{\pm 1}$, and let $\Phi$ be a positive linear map. Then, $\phi(\ms) = h(\Phi(\ms^r))$ is g-convex.
\end{theorem}
\begin{proof}
  It suffices to prove midpoint g-convexity. Since $r\in\set{\pm 1}$, $(X \gm Y)^r=X^r\gm Y^r$. Thus, applying Theorem~\ref{thm.linmap} to $\Phi$ and noting that $h$ is nondecreasing it follows that
  \begin{equation}
    \label{eq.25}
    h(\Phi(X \gm Y)^r) = h(\Phi(X^{r} \gm Y^{r})) \le h(\Phi(X^r)\gm \Phi(Y^r)).
  \end{equation}
  By assumption $h$ is g-convex, so the last inequality in~\eqref{eq.25} yields
  \begin{equation}
    \label{eq.26}
    h(\Phi(X^r)\gm \Phi(Y^r)) \le \half h(\Phi(X^r)) + \half h(\Phi(Y^r)) = \half\phi(X) + \half\phi(Y).
  \end{equation}
  Notice that if $h$ is strictly g-convex, then $\phi(\ms)$ is also strictly g-convex.
\end{proof}

\begin{example}
  Let $h=\log\det(X)$ and $\Phi(X) = B + \sum_i A_i^*XA_i$. Then, $\phi(X) = \log\det(B+\sum_iA_i^*X^rA_i)$ is g-convex. With $h(X)=\trace(X^\alpha)$ for $\alpha \ge 1$, $\trace(B + \sum_i A_i^*X^rA_i)^\alpha$ is g-convex.
\end{example}

Next, Theorem~\ref{thm.gc.sval} presents a method for creating essentially  logarithmic versions of our ``tracial'' g-convexity result Theorem~\ref{thm.trace.gc}.
\begin{theorem}
  \label{thm.gc.sval}
  If $f : \reals\to \reals$ is convex, then $\phi(\cdot):=\nlsum_{i=1}^k f(\log\lambda_i^\da(\cdot))$ is g-convex for each $1\le k \le n$. If $h : \reals \to \reals$ is nondecreasing and convex, $\phi(\cdot)=\nlsum_{i=1}^k h(|\log\lambda(\cdot)|)$ is g-convex for each $1\le k \le n$. % for $1\le k \le n$. % If $g : \reals_+\to \reals_+$ is nondecreasing and multiplicatively convex, i.e., $g(a^{1-t}b^t) \le g(a)^{1-t}g(b)^t$, then $\phi(\cdot)=\nlprod_{i=1}^k g(\lambda(\cdot))$ is log-g-convex for $1\le k \le n$.
\end{theorem}

To prove Theorem~\ref{thm.gc.sval} we will need the following majorisation.

\begin{lemma}
  \label{lem.gm.maj}
  Let $\prec_{\log}$ denote the \emph{log-majorisation} order, i.e., for $x, y \in \reals_{++}^n$ ordered nonincreasingly, we say $x \prec_{\log} y$ if $\nlprod_{i=1}^{n-1}x_i \le \nlprod_{i=1}^{n-1}y_i$ and $\prod_{i=1}^nx_i=\prod_{i=1}^ny_i$. Then, for $A, B \in \pp_n$ and $t \in [0,1]$, we have the log-majorisation between the eigenvalues:
  \begin{equation*}
    \lambda(A\gm_t B) \prec_{\log} \lambda(A^{1-t}B^t) \prec_{\log} \lambda(A^{1-t})\lambda(B^t).
  \end{equation*}
\end{lemma}
\begin{proof}
  The first majorisation follows from a recent result of~\citet{maAu12}. The second one follows easily from $\lambda_1(AB) \le \sigma_1(AB) \le \sigma_1(A)\sigma_1(B) = \lambda_1(A)\lambda_1(B)$ (the final equality holds since $A, B \in \pp_n$). Apply this inequality to the antisymmetric (Grassmann) exterior product $\wedge^k (AB)$, since $\sigma_1(\wedge^k(AB)) = \nlprod_{j=1}^k \sigma_j(AB)$ (see e.g.,~\citep[I; IV.2]{bhatia97}; then we obtain $\lambda_1(\wedge^k (AB)) \le \sigma_1(\wedge^k(AB))$. Now set $A\gets A^{1-t}$, $B\gets B^t$, and use the multiplicativity $\wedge^k(AB)=\wedge^kA\wedge^kB$ to complete the proof.
\end{proof}

%Props.~\ref{prop.gm.maj} and \ref{prop.sval.maj} we obtain the following important result.

\begin{proof}[\emph{Theorem~\ref{thm.gc.sval}}]
  From Lemma~\ref{lem.gm.maj} we have the majorisation $$\lambda(A\gm_t B) \prec_{\log} \lambda(A^{1-t}B^{t}) \prec_{\log}\lambda(A^{1-t})\lambda(B^t);$$ on taking logarithms, this majorisation may be written equivalently as
  \begin{equation}
    \label{eq:8}
    \log\lambda(A\gm_t B) \prec (1-t)\log\lambda(A)+t\log\lambda(B).
  \end{equation}
  Applying a classical result of~\citep{haLiPo} on majorisation under convex functions, from~\eqref{eq:8} we obtain the inequality
  \begin{equation*}
    \begin{split}
      \phi(A\gm_t B) &= \nlsum_{i=1}^k f(\log\lambda_i(A\gm_t B)) \le \nlsum_{i=1}^k  f\left((1-t)\log\lambda_i(A)+t\log\lambda_i(B)\right)\\
      &\le (1-t)\nlsum_{i=1}^k f(\log\lambda_i(A)) + t\nlsum_{i=1}^k f(\log\lambda_i(B))\\
      &= (1-t)\phi(A)+t\phi(B).
    \end{split}
  \end{equation*}
  Applying the Ky-Fan norm $\nlsum_{i=1}^k \sigma_i(\cdot)$---that is, the sum of top-$k$ singular values---to \eqref{eq:8}, we obtain the weak-majorisation (see e.g., \citep[II]{bhatia97} for more on majorisation):
  \begin{equation}  
    \label{eq:9}
    \sigma(\log A\gm_t B) \prec_w \sigma\left[(1-t)\log\lambda(A)+t\log\lambda(B)\right] \prec_w (1-t)\sigma(\log A)+t\sigma(\log B).
  \end{equation}
  Since $h$ is monotone and convex, \eqref{eq:9} yields g-convexity of $\nlsum_{i=1}^kh(|\log\lambda_i(\cdot)|)$. % for $1\le k \le n$.
%   Since $g$ is nondecreasing, combining it with Prop.~\ref{prop.gm.maj} we obtain
%   \begin{equation*}
%     \begin{split}
%       g(\lambda(A\gm_t B)) &\prec_{\log} g(\lambda^{1-t}(A)\lambda^t(B)) \prec_{\text{wlog}}g^{1-t}(\lambda(A))g^t(\lambda(B))\\
%       \implies\quad&\nlsum_{i=1}^k\log g(\lambda_i(A\gm_tB)) \le (1-t)\nlsum_{i=1}^k\log g(\lambda_i(A))+t\nlsum_i\log g(\lambda_i(B)),
%     \end{split}
%   \end{equation*}
% which proves the log-g-convexity of $\nlprod_{i=1}^kg(\lambda_i(\cdot))$.
\end{proof}

\begin{corr}
  \label{cor.norm.gc}
  Let $\Phi: \reals^n\to \reals_+$ be a symmetric gauge function (i.e., $\Phi$ is a norm, invariant to permutation and sign changes). Also, let $X \in \text{GL}_n(\C)$. Then, $\Phi(\sigma(\log(X^*AX))$ is g-convex.
\end{corr}
\begin{proof}
  Note that $X^*(A\gm B)X = (X^*AX)\gm(X^*BX)$; now apply Theorem~\ref{thm.gc.sval}.
\end{proof}
\begin{example}
  \label{ex.riem}
  Consider $\riem(A,X) := \frob{\log(X^{-1/2}AX^{-1/2})}$ the Riemannian distance between $A, X \in \pp_d$~\citep[Ch. 6]{bhatia07}. Since $\|\log\lambda(X^{-1/2}AX^{-1/2})\|_2=\|\sigma(\log X^{-1/2}AX^{-1/2})\|_2$, it follows from Corollary~\ref{cor.norm.gc} that $A \mapsto \riem(A,X)$ is g-convex (see also~\citep[Cor.~6.1.11]{bhatia07}).

This immediately shows that computing the Fr\'echet (Karcher) mean and median of HPD matrices (also known as geometric mean and median of HPD matrices, respectively) are g-convex optimisation problems; formally, these problems are given by
\begin{align*}
  \min_{X \succ 0}\quad &\nlsum_{i=1}^m w_i\riem(X,A_i),\qquad(\text{Geometric Median}),\\
  \min_{X \succ 0}\quad &\nlsum_{i=1}^m w_i\riem^2(X,A_i),\qquad(\text{Geometric Mean}),
\end{align*}
where $\nlsum_i w_i=1$, $w_i \ge 0$, and $A_i \succ 0$ for $1\le i \le m$. The latter problem has received great interest in the literature~\citep{moakher,bhatia07,bhatia11,ssdiv,bruno,jeuVaVa,nieBha13}, and its optimal solution is unique owing to the (strict) g-convexity of its objective. The former problem is less well-known but in some cases proves more favourable~\citep{barber,nieBha13}---again, despite the nonconvexity of the objective, its g-convexity ensures every local solution is global.
\end{example}

We conclude this section by using Lemma~\ref{lem.gm.maj} to prove the following log-convexity analogue to Theorem~\ref{thm.gc.sval} (\emph{cf.} the scalar case studied in \citep[Prop.~2.4]{nicu00}).
\begin{theorem}
  \label{thm.logc.svals}
  Let $f(x) = \nlsum_{j \ge 0} a_jx^j$ be real analytic with $a_j\ge0\ \text{for}\ j \ge 0$ and radius of convergence $R$. Then, $\phi(\cdot) = \nlprod_{i=1}^k f(\lambda_i(\cdot))$ is log-g-convex on matrices with spectrum in $(0,R)$.
\end{theorem}
\begin{proof}
  It suffices to verify that $\log\phi(A\gm B) \le \half\log\phi(A)+\half\log\phi(B)$. Since $f' \ge 0$, we have
  \begin{equation*}
    \begin{split}
      \phi(A\gm B) &= \nlprod_{i=1}^k f(\lambda_i(A\gm B)) \le \nlprod_{i=1}^k f(\lambda_i^{1/2}(A)\lambda_i^{1/2}(B))\qquad (\text{using Lemma~\ref{lem.gm.maj}})\\
      &\le \nlprod_{i=1}^k\sqrt{f(\lambda_i(A))}\sqrt{f(\lambda_i(B))}\qquad(\text{Cauchy-Schwarz on power-series of $f$})\\
      &=\sqrt{\phi(A)}\sqrt{\phi(B)}.
      % \nlsum_{j \ge 0}a_j\lambda_i^j(A\gm B) = \nlsum_{j\ge 0}a_j\nlsum_{i=1}^k\lambda_i^j(A\gm B)\\
      % &\le\nlsum_{j\ge 0}a_j\nlsum_{i=1}^k\lambda_i^{j/2}(A)\lambda_i^{j/2}(B)\qquad (\text{using Prop.~\ref{prop.gm.maj}})\\
      % &=\nlsum_{i=1}^k\nlsum_{j\ge 0} \sqrt{a_j\lambda_i^j(A)}\sqrt{a_j\lambda_i^j(B)}\\
      % &\le \nlsum_{i=1}^k\sqrt{f(\lambda_i(A))}\sqrt{f(\lambda_i(B))}\qquad(\text{Cauchy-Schwarz})\\
      % &\le\half\phi(A)+\half\phi(B)\qquad(\text{using AM-GM}).
    \end{split}
  \end{equation*}
  Taking logarithms, we see that $\phi(\cdot)$ is log-g-convex (and hence also g-convex).
\end{proof}

\begin{example}
  Some examples of $f$ that satisfy conditions of Theorem~\ref{thm.logc.svals} are $\exp$, $\sinh$ on $(0,\infty)$, $-\log(1-x)$ and $(1+x)/(1-x)$ on $(0,1)$; see~\citep{nicu00} for more examples.
\end{example}

\subsection{Multivariable g-convexity}
\label{sec:multi.gc}
We describe now an extension of g-convexity to multiple matrices; a two-variable version was also partially explored in~\citep{wie12,zhang13}, though under a different name. We begin our multivariable extension by recalling a few basic properties of the Kronecker product~\citep{vanloan00}.

\begin{lemma}
  \label{lem.kron}
  Let $A \in \reals^{m \times n}$, $B \in \reals^{p \times q}$. Then, $A\kron B := [a_{ij}B] \in \reals^{mp\times nq}$ satisfies:
  \begin{enumerate}[(i)]
    \setlength{\itemsep}{0pt}
  \item $(A\kron B)^* = A^* \kron B^*$
  \item $(A\kron B)^{-1}=\inv{A}\kron\inv{B}$
  \item Assuming that the respective products exist,  \begin{equation}\label{eq.kron.ac}(AC \kron BD) = (A\kron B)(C\kron D)\end{equation}
  \item $A \kron (B\kron C) = (A\kron B)\kron C$
  \item If $A=UD_1U^*$ and $B=VD_2V^*$ then $(A\kron B) = (U\kron V)(D_1\kron D_2)(U\kron V)^*$.
  \item Let $A, B \succeq 0$, and $t \in \reals$; then 
    \begin{equation}
      \label{eq.10}
      (A\kron B)^t = A^t \kron B^t.  
    \end{equation}
  %\item $f(A\kron B) = (U\kron V)(f(D_1) \kron f(D_2))(U\kron V)^*$
  \item If $A \succeq B$ and $C \succeq D$, then $(A\kron C) \succeq (B \kron D)$.
  \end{enumerate}
\end{lemma}
\begin{proof}
  Identities (i)--(iii) are classic; (v) follows easily from (i) and (iv), while (vi) and (vii) follow from (v); (viii) is an easy exercise. %also a standard textbook result.
\end{proof}

We will need the following easy but key result on tensor products of geometric means.
\begin{lemma}
  \label{lem.gm.kron}
  Let $A,B \in \pd_{d_1}$ and $C,D\in \pd_{d_2}$. Then,
  \begin{equation}
    \label{eq.9}
    (A\gm B)\kron(C\gm D) = (A\kron C)\gm (B \kron D).
  \end{equation}
\end{lemma}
\begin{proof}
  Denote $\gamma(X,Y) := (X^{-1/2}YX^{-1/2})^{1/2}$. Observe that
  \begin{align*}
    \gamma(A,B)\kron \gamma(C,D) &= (A^{-1/2}BA^{-1/2})^{1/2}\kron (C^{-1/2}DC^{-1/2})^{1/2}\\
    &= [(A^{-1/2}BA^{-1/2}) \kron (C^{-1/2}DC^{-1/2})]^{1/2}\\
    &= [(A\kron C)^{-1/2}(B\kron D)(A\kron C)^{-1/2}]^{1/2}\\
    &= \gamma(A\kron C, B\kron D),
  \end{align*}
  where the second equality follows from Lemma~\ref{lem.kron}-(iii), while the third one from Lemma~\ref{lem.kron}-(ii),(iii), and (vi). A similar manipulation then shows that
  \begin{align*}
    (A\gm B)\kron(C\gm D) &= (A^{1/2}\gamma(A,B)A^{1/2})\kron(C^{1/2}\gamma(C,D)C^{1/2}),\\
    &= (A^{1/2}\kron C^{1/2})(\gamma(A,B)\kron \gamma(C,D))(A^{1/2}\kron C^{1/2})\\
    &= (A\kron C)^{1/2}(\gamma(A,B)\kron\gamma(C,D))(A\kron C)^{1/2}\\
    &= (A\kron C)^{1/2}\gamma(A\kron C, B  \kron D)(A\kron C)^{1/2}\\
    &= (A\kron C)\gm (B \kron D),
  \end{align*}
  which concludes the proof.
  %, while the last one uses~\eqref{eq.kron.ac} once more.
\end{proof}

Lemma~\ref{lem.gm.kron} inductively extends to the multivariable case, so that
\begin{equation}
  \label{eq.11}
  \nlpkron_{i=1}^m (A_i\gm B_i) = \left(\nlpkron_{i=1}^m A_i\right)\gm \left(\nlpkron_{i=1}^m B_i\right).
\end{equation}
Using identity~\eqref{eq.11} we thus obtain the following multivariate analogue to Theorem~\ref{thm.gc.sval}.

\begin{theorem}
  \label{thm.multi.spec}
  Let $h$ be an increasing convex function on $\reals_+\to \reals$. Then, the map $\nlprod_{i=1}^m\trace h(X_i)$ is jointly g-convex, i.e., $\trace h(\nlpkron_{i=1}^mX_i)$ is g-convex in its variables.
\end{theorem}
\begin{proof}
Let $(A_1,B_1),\ldots,(A_m,B_m)$ be pairs of HPD matrices of arbitrary sizes (such that for each $i$, $A_i,B_i$ are of the same size. Let $j$ index the eigenvalues of the tensor product $\nlpkron_{i=1}^m (A_i\gm B_i)$. Then, starting with identity~\eqref{eq.11} we obtain
  \begin{align*}
    \lambda_j\left[\nlpkron_{i=1}^m (A_i\gm B_i)\right] &= \lambda_j\left[\left(\nlpkron_{i=1}^m A_i\right)\gm \left(\nlpkron_{i=1}^m B_i\right)\right] \le \half\lambda_j\left[\nlpkron_{i=1}^m A_i+\nlpkron_{i=1}^m B_i\right]\\
    \trace h\left(\nlpkron_{i=1}^m (A_i\gm B_i)\right)&=\sum_j h\left(\lambda_j\left[\nlpkron_{i=1}^m (A_i\gm B_i)\right]\right) 
    \le \sum_jh\left(\half\lambda_j\left[\nlpkron_{i=1}^m A_i+\nlpkron_{i=1}^m B_i\right]\right)\\
    &\le \half\sum_j h(\lambda_j(\nlpkron_{i=1}^m A_i)) + \half\sum_j h(\lambda_j(\nlpkron_{i=1}^m B_i))\\
    &=\half\trace h(\nlpkron_{i=1}^m A_i) + \half \trace h(\nlpkron_{i=1}^m B_i)\\
    &=\half\prod_{i=1}^m\trace h(A_i) + \half\prod_{i=1}^m\trace h(B_i),
  \end{align*}
  which shows the desired multivariable g-convexity of the map $\trace h(\nlpkron_{i=1}^mX_i)$.
\end{proof}

Again, using~\eqref{eq.11} we obtain the following multivariate analogue to Theorem~\ref{thm.linmap}.
\begin{theorem}
  \label{thm.multi.linmap}
  Let $(X_1,Y_1),\ldots,(X_m,Y_m)$ be pairs of HPD matrices of arbitrary sizes (such that for each $i$, $X_i,Y_i$ are of the same size). Let $\Phi_i : \Hc_i \to \Hc_i'$ be a positive linear map for each $i$, and $\Phi$ the \emph{positive multilinear  map} % from $\Hc_1\otimes \cdots \otimes \Hc_m \to \Hc_1'\otimes \cdots \otimes \Hc_m'$, such that 
defined by  $\Phi \equiv \otimes_{i=1}^m A_i \mapsto \otimes_{i=1}^m \Phi_i(A_i)$. Then, 
  \begin{equation}
    \label{eq.12}
    \Phi(\otimes_{i=1}^m (X_i\gm Y_i)) \preceq \Phi(\otimes_i X_i) \gm \Phi(\otimes_i Y_i).
  \end{equation}
\end{theorem}
\begin{proof}
  Expanding the definition of $\Phi$ we have
  \begin{align*}
    \Phi(\nlpkron_i (X_i\gm Y_i)) &= \nlpkron_i \Phi_i(X_i\gm Y_i) \preceq \nlpkron_i [\Phi_i(X_i)\gm \Phi_i(Y_i)]\\
    &= \bigl[\nlpkron_i\Phi_i(X_i)\bigr]\gm \bigl[\nlpkron_i\Phi_i(Y_i)\bigr] = \Phi(\nlpkron_i X_i)\gm \Phi(\nlpkron_i Y_i).
  \end{align*}
  The operator inequality~\eqref{eq.12} follows upon invoking Theorem~\ref{thm.linmap} and Lemma~\ref{lem.kron}-(viii).
\end{proof}

Building on Theorem~\ref{thm.multi.linmap}, we also derive a generalisation to Theorem~\ref{thm.gc}. %We state it without proof below.
\begin{theorem}
  Let $h: \kron_i\Hc_i' \to \reals$ be nondecreasing (in L\"owner order) and g-convex Let $r_i \in \set{\pm1}$ and let $\Phi: \kron_i \Hc_i \to \kron_i \Hc_i'$ be a strictly positive multilinear map. Then, $\phi(X_1,\ldots,X_m) = (h\circ \Phi)(\nlpkron_i X_i^{r_i})$ is jointly g-convex (i.e., g-convex in $X_1,\ldots,X_m$).
  % \begin{equation*}
  %   \phi(X) := h(\Phi(X^r))
  % \end{equation*}
  % is gc.
\end{theorem}
\begin{proof}
  Since $\phi$ is continuous, it suffices to establish midpoint g-convexity.
  \begin{align*}
    (h \circ \Phi)(\nlpkron_i (X_i\gm Y_i)^{r_i}) &= (h \circ \Phi)(\nlpkron_i ( X_i^{r_i}\gm Y_i^{r_i}))\\
    &\preceq h\bigl(\Phi(\nlpkron_i X_i^{r_i}) \gm \Phi(\nlpkron_i Y_i^{r_i}) )\\
    &\preceq \half\left( (h \circ \Phi)(\nlpkron_i X_i^{r_i}) + (h \circ \Phi)(\nlpkron_i Y_i^{r_i})\right)\\
    &= \half\left(\phi(X_1,\ldots,X_m)+\phi(Y_1,\ldots,Y_m)\right).
  \end{align*}
  Since $h$ is nondecreasing, using Theorem~\ref{thm.multi.linmap} the first inequality follows. The second one follows as $h$ is g-convex, which completes the proof.
\end{proof}

Using identities \eqref{eq.10} and~\eqref{eq.11} with Lemma~\ref{lem.gm.maj}  we obtain the following log-majorisations.
\begin{prop}
  \label{prop.kron.maj}
  Let $(A_i,B_i)_{i=1}^m$ be pairs of HPD matrices of compatible sizes. Then,
  \begin{equation*}
    \begin{split}
      \lambda(\nlpkron_{i=1}^m A_i\gm_t B_i) &\prec_{\log} \lambda([\nlpkron_{i=1}^mA_i]^{1-t}[\nlpkron_{i=1}^mB_i]^t),\qquad t \in [0,1]\\
      \lambda([\nlpkron_{i=1}^mA_i]^{1-t}[\nlpkron_{i=1}^mB_i]^t) &\prec_{\log} \lambda[\nlpkron_{i=1}^mA_i^{1-t}]\lambda[\nlpkron_{i=1}^mB_i^t].
 %       \lambda\left(\log\left[\nlpkron_{i=1}^m (A_i\gm_t B_i)\right]\right) &\prec (1-t)\lambda\left(\log\left[\nlpkron_{i=1}^m A_i\right]\right) + t\lambda\left(\log\left[\nlpkron_{i=1}^mB_i\right]\right).
    \end{split}
  \end{equation*}
\end{prop}

Proposition~\ref{prop.kron.maj} grants us the following multivariate analogue to Theorem~\ref{thm.gc.sval}.
\begin{theorem}
  \label{thm.kron.gc.sval}
  If $f : \reals\to \reals$ is convex, then $\phi(\cdot):=\nlsum_{j=1}^k f(\log\lambda_j(\nlpkron_{i=1}^mX_i))$ is g-convex on $\set{X_i \in \pp_n}_{i=1}^m$ for each $1\le k \le n$ . If $h : \reals \to \reals$ is nondecreasing and convex, then $\phi(\cdot)=\nlsum_{j=1}^k h(|\log\lambda_j(\nlpkron_{i=1}^mX_i)|)$ is g-convex for $1\le k \le n$.
\end{theorem}

% \begin{example}
%   TODO: Kronecker g-convex examples from Weisel etc.
% \end{example}

Theorem~\ref{thm.kron.gc.sval} brings us to the end of our theoretical results on recognising and constructing g-convex functions. We are now ready to devote attention to optimisation algorithms. In particular, we first discuss manifold optimisation~\citep{absil}  techniques in \S\ref{sec.manopt}. Then,  in \S\ref{sec.ln} we introduce a special class of functions that overlaps with g-convex functions, but not entirely, and admits simpler ``conic fixed-point'' algorithms. % in the flavor of nonlinear Perron-Frobenius theory~\citep{lemNuss12}.

\section{Manifold optimisation for g-convex functions}
\label{sec.manopt}
Since $\pp_d$ is a smooth manifold, we can use optimisation techniques based on exploiting smooth manifold structure. In addition to common concepts such as tangent vectors and derivatives along manifolds, different optimisation methods need a subset of new definitions and explicit expressions for inner products, gradients, retractions, vector transport and Hessians \citep{absil,hiai12}. %For simplicity, we consider only real symmetric positive definite matrices in this section, but continue to use $\pp_d$ to denote them.

% Without going into complications of an abstract treatment of manifolds, 
Since $\pp_d$ can be viewed as a sub-manifold of the Euclidean space $\reals^{2d^2}$, most of concepts of importance to our study can be defined by using the embedding structure of Euclidean space. The tangent space at any point is the space $\H_d$ of $d\times d$ hermitian matrices. The derivative of a function on the manifold in any direction in the tangent space is simply the embedded Euclidean derivative in that direction.

For several optimisation algorithms, two different inner product formulations were tested in \citep{jeuVaVa} for $\pp_d$. The authors observed that the intrinsic inner product leads to the best convergence speed for the tested algorithms. We too observed that the intrinsic inner product yields more than a hundred times faster convergence for our algorithms compared to the induced inner product of Euclidean space. The \emph{intrinsic inner product} of two tangent vectors at point $X$ on the manifold is given by
\begin{align}
g_{X}(\eta,\xi)=\trace (\eta X^{-1} \xi X^{-1}),\qquad \eta, \xi \in \H_d.
\label{eqn.nip}
\end{align}
This intrinsic inner product leads to geodesics of the form~\eqref{eq.14}.
Now that we have set up an inner product tensor, we can define the gradient direction as the direction of the maximum change. The inner product between the gradient vector and a vector in the tangent space is equal to the gradient of the function in that direction. If $\text{grad}^{\rm H}f(X)=\tfrac12 ( \text{grad} f(X) + (\text{grad} f(X))^{*} )$ is the hermitian part of Euclidean gradient, then the gradient in intrinsic metric is given by:
\begin{align*}
\text{grad}^{\rm HPD}f(X) =  X \text{grad}^{\rm H}f(X) X.
\end{align*}
The simplest gradient descent approach, namely steepest descent, also needs the notion of projection of a vector in the tangent space onto a point on the manifold. Such a projection is called \emph{retraction}. If the manifold is Riemannian, a particular retraction is the exponential map, i.e., moving along a geodesic. If the inner product is the induced inner product of the manifold, then the retraction is normal retraction on the Euclidean space which is obtained by summing the point on the manifold and the vector on the tangent space. The intrinsic inner product of~\eqref{eqn.nip} of the Riemannian manifold leads to the following exponential map:
\begin{align}
R_X^{\rm HPD}(\xi) =  X^{1/2} \exp ( X^{-1/2} \xi  X^{-1/2}) X^{1/2},\qquad \xi \in \H_d.
\label{eq.retr}
\end{align}
From a numerical perspective, our experiments revealed that the following equivalent representation of the retraction~\eqref{eq.retr} gives the best computational speed:
\begin{equation}
  \label{eq:19}
R_X^{\rm HPD}(\xi) = X\exp(X^{-1} \xi),\qquad \xi \in \H_d.
\end{equation}

Definitions of the gradient and retraction suffice for implementing steepest descent on $\pp_d$. For approaches such as conjugate gradients or quasi-newton methods, we need to relate the tangent vector at one point to the tangent vector at another point, i.e., we need to define \emph{vector transport}. % Vector transport takes a vector in tangent vector $\xi$ in point $X$ and moves that to another point $Y$.
 A special case of vector transport on a Riemannian manifold is parallel transport: for the induced Euclidean metric, parallel transport is simply the identity map. In order to compute the parallel transport one first needs to compute the \emph{Levi-Civita connection}. This connection is a way to compute directional derivatives of vector fields. It is a map from the Cartesian product of tangent bundles to the tangent bundle:
\begin{equation*}
\nabla : T\mathcal{M} \times  T\mathcal{M} \rightarrow T\mathcal{M},
\end{equation*}
where $T\mathcal{M}$ is the tangent bundle of manifold $\mathcal{M}$ (i.e. the space of smooth vector fields on $\mathcal{M}$).
It can be verified that for the intrinsic metric~\eqref{eqn.nip} the following connection satisfies all the needed properties (see e.g., \citep{jeuVaVa}):
\begin{align*}
  \nabla_{\zeta_X}^{\rm HPD} \xi_X = D\xi(X)[\zeta_X]-\tfrac12 (\zeta_X X^{-1} \xi_X + \xi_X X^{-1} \zeta_X),
\end{align*}
where $D \xi(X)$ denotes the classical Fr\'echet derivative of $\xi(X)$. $\xi_X$ and $\zeta_X$ are vector fields on the manifold $\H_d$. Subindex $X$ is used to discriminate a vector field from a tangent vector.

Consider $P(t)$, a vector field along the geodesic curve $\gamma(t)$. Parallel transport along a curve is given by the differential equation
\begin{align*}
D_{t} P(t) = \nabla_{\dot{\gamma}(t)} P(t) = 0,\qquad\text{s.t.}\ P(0)=\eta.
\end{align*}
% with the initial condition
% \begin{align*}
% P(0)=\eta.
%\end{align*}
For the intrinsic metric, the above equation becomes
\begin{align*}
\dot{P}(t) - \tfrac12 (\dot{\gamma}(t) X_{t}^{-1} P(t) + P(t) X_{t}^{-1}\dot{\gamma}(t)) = 0.
\end{align*}
The geodesic  passing through $\gamma(0)=X$ with $\dot{\gamma}=\xi$ is given by
\begin{align*}
\gamma(t) =  X^{1/2} \exp ( tX^{-1/2} \xi  X^{-1/2}) X^{1/2}.
\end{align*}
For $t=1$ we get the retraction~\eqref{eq.retr}. It can be shown that along the geodesic curve the following equation gives the parallel transport:
\begin{align*}
P(t) =  X^{1/2} \exp (t\tfrac12 X^{-1/2} \xi  X^{-1/2}) X^{-1/2}  \eta  X^{-1/2} \exp (t\tfrac12 X^{-1/2} \xi  X^{-1/2})  X^{1/2}.
\end{align*}
Thus, parallel transport for the intrinsic inner product is given by
\begin{align*}
\mathcal{T}_{X,Y}^{\rm HPD}(\eta) =  X^{1/2}  ( X^{-1/2} Y  X^{-1/2})^{1/2} X^{-1/2}  \eta  X^{-1/2}  ( X^{-1/2} Y  X^{-1/2})^{1/2}  X^{1/2}.
\end{align*}
It is important to note that this parallel transport can be written in a compact form that is also computationally more advantageous, namely,
\begin{equation}
\mathcal{T}_{X,Y}^{\rm HPD}(\eta) = E \eta E^*,\quad\text{where}\quad E=(YX^{-1})^{1/2}.
\label{eq.partrans}
\end{equation}

We are now ready to describe a quasi-newton method on $\pp_d$. Different algorithms such as conjugate-gradient, BFGS, and trust-region methods for the Riemmanian manifold $\pp_d$ are explained in \citep{jeuVaVa}. Here we only provide details for a limited memory version of Riemmanian BFGS (RBFGS). The RBFGS algorithm for general retraction and vector transport was originally explained in \citep{qi10} and the proof of convergence appeared in \citep{ring12}, although for a slightly different version. It was proved that for g-convex functions and with line-search that satisfies Wolfe conditions, RBFGS algorithm has a (local) superlinear convergence rate. The RBFGS algorithm can be transformed into a limited-memory RBFGS (L-RBFGS) algorithm by unrolling the update step of the approximate Hessian computation as shown in Algorithm~\ref{alg.lrbfgs}. As may be apparent from the algorithm, parallel transport and its inverse can be the computational bottlenecks. One possible speed-up is to store the matrix $E$ and its inverse in~\eqref{eq.partrans}.

\begin{algorithm}[ht]
  \caption{\small L-RBFGS}
  \label{alg.lrbfgs}
  \begin{algorithmic}
    \STATE {\bf Given:} Riemannian manifold $\Mc$ with Riemannian metric $g$; vector transport $\mathcal{T}$ on $\Mc$ with associated retraction $R$; initial value $X_0$; a smooth function $f$
    \STATE Set initial $H_{\rm diag}=1/\sqrt{g_{X_0}(\text{grad} f(X_0),\text{grad} f(X_0))}$ 
   % \STATE {\it Input:} Observations $x_1,\ldots,x_n$; function $h$
   % \STATE {\it Initialize:} $k \gets 0$; $\ms_0 \gets I_n$

    \FOR{$k=0,1,\ldots$}
    \STATE Obtain descent direction $\xi_k$ by unrolling the RBFGS method\\
    \hskip16pt$\xi_k \gets \textsc{HessMul}(-\text{grad} f(X_k), k)$
    \STATE Use line-search to find $\alpha$ s.t.\ $f(R_{X_k}(\alpha \xi_k))$ is sufficiently smaller than $f(X_k)$
    \STATE Calculate $X_{k+1}=R_{X_k}(\alpha \xi_k)$
    \STATE Define $S_k=\mathcal{T}_{X_k,X_{k+1}}(\alpha \xi_k)$
    \STATE Define $Y_k=\text{grad} f(X_{k+1})-\mathcal{T}_{X_k,X_{k+1}}(\text{grad} f(X_k))$
    \STATE Update $H_{\text{diag}}=g_{X_{k+1}}(S_k,Y_k)/g_{X_{k+1}}(Y_k,Y_k)$
    \STATE Store $Y_k$; $S_k$; $g_{X_{k+1}}(S_k,Y_k)$; $g_{X_{k+1}}(S_k,S_k)/g_{X_{k+1}}(S_k,Y_k)$; $H_{\rm diag}$
    \ENDFOR
    \RETURN $X_k$\\[8pt]
    \STATE \textbf{function} $\textsc{HessMul}(P, k)$
    \IF{$k>0$}
      \STATE  $P_{k}=P-\frac{g_{X_{k+1}}(S_k,P_{k+1})}{g_{X_{k+1}}(Y_k,S_k)}Y_k$
      \STATE   $\hat{P}=\mathcal{T}_{X_k,X_{k+1}}^{-1} \textsc{HessMul}(\mathcal{T}_{X_k,X_{k+1}} P_{k},k-1)$
      \RETURN $\hat{P} - \frac{g_{X_{k+1}}(Y_k,\hat{P})}{g_{X_{k+1}}(Y_k,S_k)}S_k + \frac{g_{X_{k+1}}(S_k,S_k)}{g_{X_{k+1}}(Y_k,S_k)}P$
      \ELSE
      \RETURN $H_{\rm diag}P$
     \ENDIF
    \STATE \textbf{end function} 
  \end{algorithmic}
\end{algorithm}

\section{Geometric optimisation for log-nonexpansive functions}
\label{sec.ln}
Though manifold optimisation is powerful and widely applicable (see e.g., the excellent toolbox~\citep{manopt}), for a special class of geometric optimisation problems we may be able to circumvent its heavy machinery in favour of potentially much simpler algorithms. 

This motivation underlies the material developed in this section, where ultimately our goal is to obtain fixed-point iterations by viewing $\pp_d$ as a convex cone instead of a Riemannian manifold. This viewpoint is grounded in nonlinear Perron-Frobenius theory~\citep{lemNuss12}, and it proves to be of practical value for our application in \S\ref{sec.ecd}. Notably, for certain problems we can obtain globally optimal solutions even without g-convexity. We believe the general conic optimisation theory developed in this section may be of wider interest.

Consider thus the following minimisation problem 
\begin{equation}
  \label{eq.minprob}
  \nlmin_{\ms \succ 0}\quad \Phi(\ms), 
\end{equation}
where $\Phi$ is a continuously differentiable real-valued function on $\pp_d$. 
Since the constraint set $\{\ms \succ 0\}$ is an open subset of a Euclidean space, the first-order optimality condition for~\eqref{eq.minprob} is similar to that of unconstrained optimisation. A point $S^*$ is a candidate local minimum of $\Phi$ only if its gradient at this point is zero, that is,
\begin{equation}
\nabla \Phi(\ms^*) =0.
  \label{eq.mingrad}
\end{equation}

The nonlinear (matrix) equation~\eqref{eq.mingrad} could be solved using numerical techniques such as Newton's method. But, such approaches can be computationally more demanding than the original optimisation problem, especially because they involve the (inverse of) the second derivative $\nabla^2\Phi$ at each iteration. We propose to exploit a fixed-point iteration that offers a simpler method for solving~\eqref{eq.mingrad}. \emph{More importantly,} the fixed-point technique allows one to show that under certain conditions the solution to~\eqref{eq.mingrad} is unique, and therefore potentially a global minimum (essentially, if the global minimum is attained, then it must be this unique stationary point). %the fixed-point approach has the benefit of not relying on convexity or g-convexity of the objective function $\Phi$, while still yielding global optima.

Assume therefore that~\eqref{eq.mingrad} is rewritten as the fixed-point equation
\begin{equation}
  S^*=\Gc(S^*).
  \label{eq.pz}
\end{equation}
Then, a fixed-point of the map $\Gc : \pp_d \to \pp_d$ is a potential solution (since it is a stationary point) to the minimisation problem~\eqref{eq.minprob}. The natural question is how to find such a  fixed-point, and starting with a feasible $\ms_0 \succ 0$, whether it suffices to perform the Picard iteration
\begin{equation}
  \label{eq:iterform}
  \ms_{k+1} \gets \Gc(\ms_k),\quad k=0,1,\ldots.
\end{equation}
Iteration~\eqref{eq:iterform} is (usually) \emph{not} a fixed-point iteration when cast in a normed vector space---the conic geometry of $\pp_d$ alluded to previously suggests that it might be better to analyse the iteration using a non-vectorial metric.

We provide below a class of sufficient conditions ensuring convergence of~\eqref{eq:iterform}. Therein, the correct metric space in which to study convergence is neither the Euclidean (or Banach) space $\reals^n$ nor the Riemannian manifold $\pp_d$ with distance~\eqref{eq.metr}. Instead, a conic metric proves more suitable, namely, the  Thompson part metric, an object of great interest in nonlinear Perron-Frobenius theory~\citep{lemNuss12,leeLim}. %, and it proves very useful (akin to the Hilbert projective metric) for analysing problems on convex cones.

Our sufficient conditions stem from the following key definition.
\begin{defn}[Log-nonexpansive]
  Let $f: (0,\infty)\to (0,\infty)$. We say $f$ is \emph{log-nonexpansive} (LN) on a compact interval $I \subset (0,\infty)$ if there exists a constant $0 \le q \le 1$ such that
  \begin{equation}
    \label{eq.17}
    |\log f(t) - \log f(s)| \le q|\log t - \log s|,\quad \forall s, t \in I.
  \end{equation}
  If $q<1$, we say $f$ is \emph{$q$-log-contractive}. If for every $s\neq t$ it holds that
  \begin{equation*}
    |\log f(t) - \log f(s)| < |\log t - \log s|,\quad \forall s,t\quad s\neq t,
  \end{equation*}
  we say $f$ is log-contractive. %; notice the absence of the contraction coefficient $q$.  %In other words, $f = e^\psi$, for $\psi$ nonexpansive on the interval $I$.
\end{defn}

%Our aim is to study fixed-point iterations involving log-contractive functions. It turns out that the correct metric space in which to study these fixed-point iterations is neither the Euclidean space $\reals^n$ nor the manifold $\pp_d$ with its Riemannnian distance~\eqref{eq.metr}. Instead, a close relative of the Riemannian distance proves more suitable: namely, the  Thompson part metric.

We use log-nonexpansive functions in a concrete optimisation task in Section~\ref{sec.example}.  The proofs therein rely on core properties of the Thompson metric and contraction maps in the associated metric space---we cover requisite background in Section~\ref{sec.thompson}. % , we give a short overview of the Thompson metric and some of its properties along with theorems on contractive maps crucial for our analysis in the subsequent subsection.
The content of Section~\ref{sec.thompson} is of independent interest as the theorems therein provide techniques for establishing contractivity (or nonexpansivity) of nonlinear maps from $\pp_d$ to $\pp_k$.

%} %ENDCOLOR
\subsection{Thompson metric and contractive maps}
\label{sec.thompson}
On $\pp_d$, the \emph{Thompson metric} is defined as (\emph{cf.} $\riem$ which uses $\frob{\cdot}$)
\begin{equation}
\label{thom}
  \thom(X,Y) :=  \norm{\log (Y^{-1/2}XY^{-1/2})},
\end{equation}
where $\norm{\cdot}$ is the usual operator norm (largest singular value), and `log' is the matrix logarithm.  Let us recall some core (known) properties of~\eqref{thom}---for details please see~\citep{leeLim,lemNuss12,limPal12}. 

\begin{subequations}
  \begin{prop}
    \label{prop.thom}
    Unless noted otherwise, all matrices are assumed to be HPD.
    \begin{align}
      \label{eq:4}
      \thom(\inv{X},\inv{Y}) &\quad=\quad \thom(X,Y)\\
      \label{eq:5}
      \thom(B^*XB,B^*YB)     &\quad=\quad \thom(X,Y),\qquad B \in\text{GL}_n(\C)\\
      \label{eq:6}
      \thom(X^t,Y^t)         &\quad\le\quad |t|\thom(X,Y),\qquad\text{for}\ t\in [-1,1]\\
      \label{eq:7}
      \thom\Bigl(\nlsum_i w_iX_i, \nlsum_i w_iY_i \Bigr) &\quad\le\quad \max_{1 \le i \le m} \thom(X_i,Y_i),\qquad w_i \ge 0, w \neq 0\\
      \label{eq:10}
      \thom(X+A, Y+A) &\quad\le\quad
      \frac{\alpha}{\alpha+\beta}\thom(X,Y),\qquad A \succeq 0,
    \end{align}
    where $\alpha=\max\set{\norm{X},\norm{Y}}$ and
    $\beta=\lambda_{\min}(A)$.
  \end{prop}
\end{subequations}

We prove now a powerful refinement to~\eqref{eq:5}, which shows contraction under ``compression.''
\begin{theorem}
  \label{thm.comp}
  Let $X \in \C^{d \times p}$, where $p \le d$ have full column rank. Let $A$, $B \in \pp_d$. Then,
  \begin{equation}
    \label{eq.21}
    \thom(X^*AX,X^*BX) \le \thom(A,B).
  \end{equation}
\end{theorem}
\begin{proof}
  % Write $X$ using its full-SVD $X=U\Sigma V^*$, and let $\Sigma_+$ denote the strictly positive part of $\Sigma$ (note $V$ is invertible). Then,
  % \begin{equation*}
  %   \begin{split}
  %     \thom(X^*AX,X^*TBX) &= \thom(V\Sigma U^*AU\Sigma V^*, V\Sigma U^*BU\Sigma  V^*)\\
  %     &= \thom(\Sigma_+U_p^*AU_p\Sigma_+,\Sigma_+U_p^*BU_p\Sigma_+)\\
  %     &= \thom(U_p^*AU_p,U_p^*BU_p).
  %   \end{split}
  % \end{equation*}
  % It suffices therefore to where the final inequality follows from Lemma~\ref{lem.partial.isom}.
%   \begin{equation}
%     \label{eq.2dup}
%     \thom(U^*AU,U^*BU) \le \thom(A,B).
%   \end{equation}
% \end{lemma}
  Let $A_C = X^*AX$ and $B_C = X^*BX$ denote the ``compressions'' of $A$ and $B$, respectively; these compressions are invertible since $X$ is assumed to have full column rank. The largest generalised eigenvalue of the pencil $(A,B)$ is given by
  \begin{equation}
    \label{eq.3dup}
    \lambda_1(A,B) := \lambda_1(\inv{A}B) = \max_{x \neq 0} \frac{x^*Bx}{x^*Ax}.
  \end{equation}
  Starting with~\eqref{eq.3dup} we have the following relations:
  \begin{align*}
    \lambda_1(\inv{A}B) &= \lambda_1(A^{-1/2}BA^{-1/2})\quad=\quad\max_{x \neq 0}\frac{x^*A^{-1/2}BA^{-1/2}x}{x^*x}\\
    &=\max_{w \neq 0} \frac{w^*Bw}{(A^{1/2}w)^*(A^{1/2}w)}\quad=\quad\max_{w\neq 0}\frac{w^*Bw}{w^*Aw}\\
    &\ge \max_{w=Xp, p\neq 0}\frac{w^*Bw}{w^*Aw}\quad=\quad\max_{p\neq 0}\frac{p^*X^*BXp}{p^*X^*AXp}\\
    &=\max_{p \neq 0}\frac{p^*B_Cp}{p^*A_Cp}\quad=\quad\lambda_1(A_C^{-1}B_C) = \lambda_1(A_C^{-1/2}BA_C^{-1/2}).
  \end{align*}
  Similarly, we can show that $\lambda_1(\inv{B}A) = \lambda_1(B^{-1/2}AB^{-1/2}) \ge \lambda_1(B_C^{-1/2}A_CB_C^{-1/2})$. Since $A$, $B$ and the matrices $A_C$, $B_C$ are all positive, we may conclude
  \begin{equation}
    \label{eq.4}
    \max \set{\log \lambda_1(A_U^{-1}B_U), \log \lambda_1(B_U^{-1}A_U)} \le \max\set{\lambda_1(A^{-1}B), \log \lambda_1(B^{-1}A)},
  \end{equation}
  which is nothing but the desired claim $\thom(X^*AX,X^*BX) \le \thom(A,B)$.
\end{proof}

%%% This theorem and its corollaries can be removed from the journal submission
Theorem~\ref{thm.comp} can be extended to encompass more general ``compression'' maps, namely to those defined by operator monotone functions, a class that enjoys great importance in matrix theory---see e.g.,~\citep[Ch.~V]{bhatia97} and \citep{bhatia07}.
\begin{theorem}
  \label{thm.comp2}
  Let $f$ be an operator monotone (i.e., if $X \preceq Y$, then $f(X) \preceq f(Y)$) function on $(0,\infty)$ such that $f(0) \ge 0$. Then,
  \begin{equation}
    \label{eq:17}
    \thom(f(X), f(Y)) \le \thom(X,Y),\qquad X, Y \in \pp_d.
  \end{equation}
\end{theorem}
\begin{proof}
  If $f$ is operator monotone with $f(0) \ge 0$, then it admits the integral representation~\citep[(V.53)]{bhatia97}
  \begin{equation}
    \label{eq:16}
    f(t) = \gamma + \beta t + \int_0^\infty \frac{\lambda t}{\lambda + t}d\mu(\lambda),
  \end{equation}
  where $\gamma = f(0)$, $\beta \ge 0$, and $d\mu(t)$ is a nonnegative measure. Using~\eqref{eq:16} we get
  \begin{equation*}
    f(A) = \gamma I + \beta A + \int_0^\infty (\lambda A)(\lambda I + A)^{-1}d\mu(\lambda) =: \gamma I + \beta A + M(A).
  \end{equation*}
  Similarly, we obtain $f(B) = \gamma I + \beta B + M(B)$. Now, consider at first
  \begin{align*}
    \thom(M(A),M(B)) &= \thom(\sint \lambda A(\lambda I + A)^{-1}d\mu(t), \sint \lambda A(\lambda I + A)^{-1}d\mu(t))\\
    &\stackrel{}{\le} \max_{\lambda} \thom(\lambda A(\lambda I + A)^{-1}, \lambda B(\lambda I + B)^{-1})\\
    &\stackrel{}{\le} \max_{\lambda} \thom( (\lambda A^{-1}+I)^{-1}, (\lambda B^{-1}+I)^{-1})\\
    &=\max_{\lambda}\thom( I + \lambda A^{-1}, I + \lambda B^{-1})\\
    &\stackrel{}{\le} \max_{\lambda}\frac{\bar{\alpha}}{\bar{\alpha}+1}\thom(\lambda A^{-1}, \lambda B^{-1}),\qquad \bar{\alpha} := \max\lbrace\norm{A^{-1}}, \norm{B^{-1}}\rbrace,\\
    &=\frac{\bar{\alpha}}{\bar{\alpha}+1}\thom(A,B) < \thom(A,B).
  \end{align*}
  Next, defining $\alpha := \max\{\|\beta A + M(A)\|, \|\beta B + M(B)\|\}$, we can use the above contraction to help prove contraction for the map $f$ as follows:
  \begin{align*}
    \thom(f(A),f(B)) &= \thom(\gamma I + \beta A + M(A), \gamma I + \beta B + M(B))\\
    &\stackrel{}{\le} \frac{\alpha}{\alpha+\gamma}\thom(\beta A + M(A), \beta B + M(B)),\\
    &\stackrel{}{\le} \frac{\alpha}{\alpha+\gamma}\max\left\lbrace\thom(\beta A, \beta B), \thom(M(A), M(B)) \right\rbrace\\
    &\le \frac{\alpha}{\alpha+\gamma}\thom(A, B).
  \end{align*}
  Moreover, for $A\neq B$ the inequality is strict if $f(0) > 0$.
\end{proof}
\begin{example}
  Let $X \in \C^{d\times k}$, let $f=t^r$ for $t \in (0,\infty)$ and $r \in (0,1)$. Then,
  \begin{equation*}
    \begin{split}
      \thom((X^*AX)^r, (X^*BX)^r) &\le \thom(A,B),\qquad\forall A, B \in \pp_d,\\
      \thom(X^*A^rX, X^*B^rX) &\le \thom(A,B),\qquad\forall A, B \in \pp_d.
    \end{split}
  \end{equation*}
\end{example}

Theorem~\ref{thm.comp} and Theorem~\ref{thm.comp2} together yield the following general result.

\begin{corr}
  \label{cor.comp}
  Let $\Phi: \pp_d \to \pp_k$ ($k \le d$), and $\Psi: \pp_k \to \pp_r$ ($r \le k$) be completely positive (see e.g.,~\citep[Ch.~3]{bhatia07}) maps. Then,
  \begin{align}
    \label{eq:14}
    \thom(f(\Phi(X)), f(\Phi(Y))) &\le \thom(X,Y),\qquad X, Y \in \pp_d,\\
    \label{eq:15}
    \thom(\Psi(f(X)), \Psi(f(Y))) &\le \thom(X,Y),\qquad X, Y \in \pp_k.
  \end{align}
\end{corr}
\begin{proof}
  We prove~\eqref{eq:14}; the proof of \eqref{eq:15} is similar, hence omitted. From Theorem~\ref{thm.comp2} it follows that $\thom(f(\Phi(X)), f(\Phi(Y))) \le \thom(\Phi(X),\Phi(Y))$. Since $\Phi$ is completely positive, it follows from a result of \citet{choi75} and \citet{kraus71} that there exist matrices $V_j \in \C^{d\times k}$, $1\le j \le dk$, such that
  \begin{equation*}
    \Phi(X) = \nlsum_{i=1}^{nk}V_j^*XV_j\qquad X \in \pp_d.
  \end{equation*}
  Theorem~\ref{thm.comp} and property~\eqref{eq:7} imply that $\thom(\Phi(X),\Phi(Y)) \le \thom(X,Y)$, which proves~\eqref{eq:14}.
\end{proof}

%\noindent\paragraph{Thompson log-nonexpansive maps}
\subsubsection{Thompson log-nonexpansive maps}
Let $\Gc$ be a map from $\Xc \subseteq \pp_d \to \Xc$. Analogous to~\eqref{eq.17}, we say $\Gc$ is (Thompson) \emph{log-nonexpansive} if
\begin{equation*}
  \thom(\Gc(X),\Gc(Y)) \le \thom(X,Y),\qquad \forall X, Y \in \Xc;
\end{equation*}
the maps is called \emph{log-contractive} if the inequality is strict. We present now a key result that justifies our nomenclature and the analogy to~\eqref{eq.17}: it shows that the sum of a log-contractive map and a log-nonexpansive map is log-contractive. This behaviour is a striking feature of the nonpositive curvature of $\pp_d$; such a result does \emph{not} hold in normed vector spaces. 

\begin{theorem}
\label{thm.thomsum}
Let $\Gc$ be a log-nonexpansive map and $\Fc$ be a log-contractive one. Then, their sum $\Gc+\Fc$ is log-contractive.
\end{theorem} 
\begin{proof}
  We start by writing Thompson metric in an alternative form~\citep{lemNuss12}:
\begin{equation}
  \label{eq.thom.def2}
\thom(A,B) =\max \{\log W(A/B) , \log W(B/A) \},
\end{equation} 
where $W (A/B) := \inf\{\lambda>0, A \preceq \lambda B\}$. 
Let $\lambda=\exp(\thom(X,Y))$; then it follows that $X\preceq \lambda Y$.
% \begin{align*}
%  X  \quad &\preceq \quad \lambda Y 
%  \end{align*}
 Since $\Gc$ is nonexpansive in $\thom$, using~\eqref{eq.thom.def2} it further follows that
 \begin{align*}
  \Gc(X)  \quad &\preceq \quad \lambda \Gc(Y),
  \end{align*}
  and $\Fc$ is log-contractive map, we obtain the inequality
 \begin{align*}
  \Fc(X)  \quad &\prec \quad \lambda^{t} \Fc(Y),\qquad\text{where } t \le 1.
  \end{align*}  
  Write $\Hc := \Gc + \Fc$; then, we have the following inequalities:
\begin{align*}
 \Hc(X)  \quad &\prec \quad \lambda \Hc(Y) + (\lambda^{t}-\lambda) \Fc(Y) \\
 \Hc(Y)^{-1/2}  \Hc(X)  \Hc(Y)^{-1/2} \quad &\prec \quad \lambda I +  (\lambda^{t}-\lambda)  \Hc(Y)^{-1/2} \Fc(Y)  \Hc(Y)^{-1/2} \\
 \Hc(Y)^{-1/2}  \Hc(X)  \Hc(Y)^{-1/2} \quad &\prec \quad \lambda I +  (\lambda^{t}-\lambda) \lambda_{\min}(\Hc(Y)^{-1/2} \Fc(Y)  \Hc(Y)^{-1/2}) I,
\end{align*}
As $\lambda_{\max}(\Hc(Y)^{-1/2}  \Hc(X)  \Hc(Y)^{-1/2}) >
\lambda_{\max}(\Hc(X)^{-1/2}  \Hc(Y)  \Hc(X)^{-1/2})$, using \eqref{eq.thom.def2} we obtain
\begin{equation}
\label{eq:20}
\thom(\Hc(X),\Hc(Y)) < \thom(X,Y) + \log \bigl( 1+ \lambda_{\min}(\Hc(Y)^{-1/2} \Fc(Y)  \Hc(Y)^{-1/2}) \left [ \lambda^{t-1}-1 \right ] \bigr).
\end{equation}
We also have the following eigenvalue inequality
\begin{equation}
\label{eq:21}
\lambda_{\min}(\Hc(Y)^{-1/2} \Fc(Y)  \Hc(Y)^{-1/2}) \leq \frac{\lambda_{\min}(\Fc(Y))}{\lambda_{\max}(\Gc(Y))+\lambda_{\min}(\Fc(Y))}.
\end{equation}
Combining inequalities \eqref{eq:20} and \eqref{eq:21} we see that
\begin{equation}
\label{eq:22}
  \thom(\Hc(X),\Hc(Y))
< \thom(X,Y)+\log \bigl( 1 + \tfrac{ \lambda_{\min}(\Fc(Y))}
{\lambda_{\max}(\Gc(Y))+\lambda_{\min}(\Fc(Y))}
\bigl[ \exp(\thom(X,Y))^{t-1}-1 \bigr] \bigr).
\end{equation}
Similarly, since $\lambda_{\max}\bigl(\Hc(Y)^{-1/2}  \Hc(X)  \Hc(Y)^{-1/2}\bigr) <  \lambda_{\max}\bigl(\Hc(X)^{-1/2}  \Hc(Y)\Hc(X)^{-1/2}\bigr)$,
we also obtain the bound (notice we now have $\Fc(X)$ instead of $\Fc(Y)$)
\begin{equation}
  \label{eq:23}
\thom(\Hc(X),\Hc(Y)) < \thom(X,Y)+\log \bigl( 1+ \tfrac{\lambda_{\min}(\Fc(X))}{\lambda_{\max}(\Gc(X))+\lambda_{\min}(\Fc(X))} \bigl[ \exp(\thom(X,Y))^{t-1}-1 \bigr] \bigr).
\end{equation}
Combining \eqref{eq:22} and \eqref{eq:23} into a single inequality, we get
\begin{multline*}
\thom(\Hc(X),\Hc(Y))\\ < \thom(X,Y) +\log \bigl( 1+ \tfrac{\lambda_{\min}(\Fc(X),\Fc(Y))}{\lambda_{\max}(\Gc(X),\Gc(Y))+\lambda_{\min}(\Fc(X),\Fc(Y))} \bigl[ \exp(\thom(X,Y))^{t-1}-1 \bigr] \bigr).
\end{multline*}
As the second term is $\le 0$, the inequality is strict, proving log-contractivity of $\Hc$.
\end{proof}

Using log-contractivity we can finally state our main result for this section.
\begin{theorem}
  \label{thm.ln.convg}
  If $\Gc$ is log-contractive  and equation~\eqref{eq.pz} has a solution, then this solution is unique and iteration \eqref{eq:iterform} converges to it.
\end{theorem}
\begin{proof}
  If~\eqref{eq.fixediter} has a solution, then from a theorem of \citet{edel62}, it follows that the log-contractive map $\Gc$ yields iterates that stay within a compact set and converge to a unique fixed point of $\Gc$. This fixed-point is positive definite by construction (starting from a positive definite matrix, none of the operations in~\eqref{eq.fixediter} violates positivity). Thus, the unique solution is positive definite.
\end{proof}

\subsection{Example of log-nonexpansive optimisation}
\label{sec.example}
%{\color{green}
To illustrate how to exploit log-nonexpansive functions for optimisation, let us consider the following minimisation problem
%} %ENDCOLOR
%directly move to our application that originally motivated us to develop the material within this paper. To compute a maximum likelihood estimate we equivalently consider the  minimisation problem:
\begin{equation}
  \label{eq.main2}
  \nlmin_{\ms \succ 0}\quad \Phi(\ms) := \tfrac12 n\log\det(\ms) - \nlsum_i \log\varphi(x_i^T\inv{\ms}x_i),
\end{equation}
which arises in maximum-likelihdood estimation of ECDs (see Section~\ref{sec.ecd} for further examples and details) and also M-estimation of the scatter matrix~\citep{kent91}.

The first-order necessary optimality condition for~\eqref{eq.main2} stipulates that a candidate solution $S \succ 0$ must satisfy
\begin{equation}
  \label{eq.fonc2}
  \frac{\partial\Phi(\ms)}{\partial\ms} = 0\quad\Longleftrightarrow\quad \tfrac12 n\inv{\ms} + \sum_{i=1}^n\frac{\varphi'(\vx_i^T\inv{\ms}\vx_i)}{\varphi(\vx_i^T\inv{\ms}\vx_i)}\inv{\ms}\vx_i\vx_i^T\inv{\ms} = 0.
\end{equation}
Defining $h \equiv -\varphi'/\varphi$, \eqref{eq.fonc2} may be rewritten more compactly in matrix notation as the equation
\begin{equation}
  \label{eq.fixediter}
  \ms = \tfrac{2}{n}\nlsum_{i=1}^n \vx_ih(\vx_i^T\inv{\ms}\vx_i)\vx_i^T = \tfrac{2}{n}\mx h(\md_{\ms})\mx^T,
\end{equation}
where $h(\md_{\ms}) := \Diag(h(x_i^T\inv{\ms}x_i))$, and $\mx=[x_1,\ldots,x_m]$.
We then solve the nonlinear equation~\eqref{eq.fixediter} via a fixed-point iteration. Introducing  the nonlinear map $\Gc : \pp_d \to \pp_d$ that maps $\ms$ to the right hand side of~\eqref{eq.fixediter}, we use fixed-point iteration~\eqref{eq:iterform} to find the solution. In order to show that the Picard iteration converges (to the unique fixed-point), it is enough to show that $\Gc$ is log-contractive (see Theorem~\ref{thm.ln.convg}). The following proposition gives sufficient condition on $h$, under which the map is log-contractive. 
%{\color{green}

%} %ENDCOLOR
%More specifically, we solve~\eqref{eq.fixediter} via a fixed-point iteration. Introducing  the nonlinear map $\Gc : \pp_d \to \pp_d$ that maps $\ms$ to the right hand side of~\eqref{eq.fixediter}, we start with a feasible $\ms_0 \succ 0$ and simply perform the iteration
%\begin{equation}
%  \label{eq:iterform}
%  \ms_{k+1} \gets \Gc(\ms_k),\quad k=0,1,\ldots.
%\end{equation}
%Iteration~\eqref{eq:iterform} is (usually) \emph{not} a fixed-point iteration when cast in a normed vector space. But the conic geometry of $\pp_d$ alluded to previously suggests that it might be better to analyse the iteration using a non-Euclidean metric.

\begin{prop}
\label{prop.ln}
Let $h$ be log-nonexpansive. Then, the map $\Gc$ in~\eqref{eq:iterform} is log-nonexpansive. Moreover, if $h$ is log-contractive, then $\Gc$ is log-contractive.
\end{prop}
\begin{proof}
  Let $\ms, \mr \succ 0$ be arbitrary. Then, we have the following chain of inequalities
  \begin{align*}
    \thom(\Gc(\ms),\Gc(\mr)) &=\quad\thom\bigl(\tfrac{2}{n}\mx h(\md_{\ms})\mx^T,\ \tfrac{2}{n}\mx h(\md_{\mr})\mx^T\bigr)\\
    &\le\quad\thom\bigl(h(\md_{\ms}), h(\md_{\mr}) \bigr)\quad\le\quad\max_{1\le i \le n}\thom\bigl(h(\vx_i^T\inv{\ms}\vx_i), h(\vx_i^T\inv{\mr}\vx_i)\bigr)\\
    &\le\quad\max_{1\le i \le n}\thom\bigl(\vx_i^T\inv{\ms}\vx_i, \vx_i^T\inv{\mr}\vx_i\bigr)\quad\le\quad\thom\bigl(\inv{\ms},\inv{\mr}\bigr) = \thom(\ms,\mr).
  \end{align*}
  The first inequality follows from~\eqref{eq:5} and Theorem~\ref{thm.comp}; the second inequality follows since $h(\md_{\ms})$ and $h(\md_{\mr})$ are diagonal; the third follows from~\eqref{eq:7}; the fourth from another application of Theorem~\ref{thm.comp}, while the final equality is via~\eqref{eq:4}. This proves log-nonexpansivity (i.e., nonexpansivity in $\thom$). If in addition
  $h$ is log-contractive and $\ms \neq \mr$, then the second inequality above is strict, that is,
  \begin{equation*}
    \thom(\Gc(\ms),\Gc(\mr)) < \thom(\ms,\mr) \quad \forall \ms,\mr \quad \text{and}\quad \ms \neq \mr.
  \end{equation*}
\end{proof}

If $h$ is merely log-nonexpansive (not log-contractive), it is still possible to show uniqueness of~\eqref{eq.fixediter} up to a constant. Our proof depends on the compression property of $\thom$ proved in Theorem~\ref{thm.comp}.

\begin{theorem}
\label{thm.thomln}
 Let the data $\Xc=\set{x_1,\ldots,x_n}$ span the whole space. If $h$ is LN, and $\ms_1\neq \ms_2$ are solutions to equation~\eqref{eq.fixediter}, then iteration \eqref{eq:iterform} converges to a solution, and $\ms_1 \propto \ms_2$.
\end{theorem} 
\begin{proof}
Without loss of generality assume that $\ms_1=I$. Let $\ms_2 \neq c I$. Theorem~\ref{thm.comp} implies that
\begin{align*}
  &\thom\bigl(\vx_i h(\vx_i^T\inv{\ms_2}\vx_i) \vx_i^T, \vx_i h(\vx_i^T\inv{\ms_1}\vx_i) \vx_i \bigr)\\
    &\le\quad \thom\bigl(h(\vx_i^T\inv{\ms_2}\vx_i), h(\vx_i^T \vx_i)\bigr) \le\quad \thom\bigl(\vx_i^T\inv{\ms_2}\vx_i, \vx_i^T\vx_i\bigr) =\quad \left|\log \tfrac{\vx_i^T\inv{\ms_2}\vx_i}{\vx_i^T\vx_i}  \right |.
\end{align*}
As per assumption, the data span the whole space. Since $\ms_2 \neq c I$, we can find $\vx_1$ such that
\begin{equation*}
\left |\log  \tfrac{\vx_1^T\inv{\ms_2}\vx_1}{\vx_1^T\vx_1}  \right | \quad<\quad \thom(\ms_2,I).
\end{equation*}
Therefore, we obtain the following inequality for point $\vx_1$:
\begin{equation}
\thom\bigl(\vx_1 h(\vx_i^T\inv{\ms_2}\vx_1) \vx_1^T, \vx_1 h(\vx_1^T\inv{\ms_1}\vx_1) \vx_1 \bigr) \quad<\quad \thom(\ms_2,\ms_1). 
\end{equation}
Using Proposition~\ref{prop.ln} and invoking Theorem~\ref{thm.thomsum}, it then follows that
\begin{equation*}
\thom(\Gc(\ms_2),\Gc(\ms_1)) \quad < \quad \thom(\ms_2,\ms_1).
\end{equation*}
But this means that $\ms_2$ cannot be a solution to~\eqref{eq.fixediter}, a contradiction. Therefore, $\ms_2 \propto \ms_1$.
\end{proof}

%\noindent\paragraph{Computational efficiency}
\subsubsection{Computational efficiency}
%{ \color{green}
So far we did not address computational efficacy of the fixed-point algorithm. The rate of convergence depends heavily on the contraction factor, and as we will see in the experiments, without further care one obtains poor contraction factors that can lead to a very slow convergence. We briefly discuss below a useful speedup technique that seems to have a dramatic impact on the empirical convergence speed (see Figure~\ref{fig.dim}).

At the fixed point $\ms^*$ we have $\Gc(\ms^*)=\ms^*$, or equivalently for a new map $\Mc$ we have $$\Mc(\ms^*):={\ms^*}^{-1/2}\Gc(\ms^*){\ms^*}^{-1/2}=I.$$ Therefore, one way to analyse the convergence behaviour is to assess how fast $\Mc(\ms_k)$ converges to identity.  Using the theory developed beforehand, it is easy to show that
\begin{equation*}
\thom(M(\ms_{k+1}),I) \quad \leq \quad \eta \thom(M(\ms_k),I),
\end{equation*}
where $\eta$ is the contraction factor between $S_k$ and $S_{k+1}$, so that
\begin{equation*}
\thom(\Gc(\ms_{k+1}),\Gc(\ms_k)) < \eta \thom(\ms_{k+1},\ms_k).
\end{equation*}
To increase the convergence speed we may replace $\ms_{k+1}$ by its scaled version $\alpha_k\ms_{k+1}$ such that
\begin{equation*}
\thom(\Mc(\alpha_k\ms_{k+1}),I) \leq \thom(\Mc(\ms_{k+1}),I).
\end{equation*}
One can do a search to find a good $\alpha_k$. Clearly, the sequence $\ms_{k+1}=\alpha_k \Gc(\ms_k)$ converges at a faster pace. We will see in the numerical results section that scaling with $\alpha_k$ has a remarkable effect on the convergence speed. An intuitive reasoning why this happens is that the additional scaling factor can resolve the problematic cases where the contraction factor become small. These problematic cases are those where both the smallest and the largest eigenvalues of $\Mc(\ms_k)$ become smaller (or larger) than one, whereby the contraction factor (for $\Gc$) becomes small, which may lead to a very slow convergence. The scaling factor, however, makes the smallest eigenvalues of $\Mc(\ms_k)$ always smaller and its largest eigenvalue larger than one. One way to avoid the search is to choose $\alpha_k$ such that $\text{trace}(\Mc(\ms_{k+1}))=d$---though with a small caveat: empirically this simple choice of $\alpha_k$ works very well, but our  convergence proof does not hold anymore. Extending our convergence theory to incorporate this specific choice of scaling $\alpha_k$ is a part of our future work. In all simulations in the result section $\alpha_k$ is selected by ensuring $\text{trace}(\Mc(\ms_{k+1}))=d$.
%It is our future aim to investigate if one can give a proof for this case of scaling parameter and to see how it affects the convergence rate.
%} %ENDCOLOR
\section{Application to Elliptically Contoured Distributions}
\label{sec.ecd}

In this section we present details for a concrete application of conic geometric optimisation:  mle for ECDs~\citep{cambanis81,gupta99,muirhead82}. We use ECDs as a platform for illustrating geometric optimisation because ECDs are widely important (see e.g., the survey~\citep{ollila11}), and are instructive in illustrating our theory. % but also because several of key ideas were originally discovered while addressing maximum likelihood estimation for ECDs. We note however, as a pleasant byproduct, the resulting theory applies more widely to any problem that possesses the requisite geometric structure.

First, some basics. If an ECD has density on $\reals^d$, it assumes the form\footnote{For simplicity we describe only mean zero families; the extension to the general case is easy.}
\begin{equation}
  \label{eq.1}
  \forall\ x \in \reals^d,\qquad \Escr_\varphi(x;\ms) \propto \det(\ms)^{-1/2}\varphi(\vx^T\ms^{-1}\vx),
\end{equation}
where $\ms \in \pp_d$ is the scatter matrix and $\varphi : \reals \to \reals_{++}$ is the \emph{density generating function} (dgf). If the ECD has finite covariance, then the scatter matrix is proportional to the covariance matrix~\citep{cambanis81}.

\begin{example}
  Let $\varphi(t)=e^{-t/2}$; then, \eqref{eq.1} reduces to the multivariate Gaussian density. For 
  \begin{equation}
    \label{eq.24}
    \varphi(t)=t^{\alpha-d/2} \exp \bigl( -(t/b)^\beta\bigr),
  \end{equation}
where $\alpha$, $b$, $\beta >0$ are fixed, density~\eqref{eq.1} yields the rich class called \emph{Kotz-type distributions} that have  powerful modelling abilities~\citep[\S3.2]{kotz}; they include as special cases multivariate power exponentials, elliptical gamma, multivariate W-distributions, for instance. Other examples include multivariate student-t, multivariate logistic, and Weibull dgfs (see \S\ref{sec.ecd.gc}).
% \note{Mention student-t, gamma, weibull, etc. as other special cases, though without listing dgfs.}
\end{example}

\subsection{Maximum likelihood parameter estimation}

Let $(\vx_1,\ldots,\vx_n)$ be i.i.d.\ samples from an ECD~$\Escr_\varphi(\ms)$. Ignoring constants, the log-likelihood is
\begin{equation}
  \label{eq.lh}
  \Lc(\vx_1,\ldots,\vx_n; \ms) = -\tfrac12 n\log\det\ms + \nlsum_{i=1}^n\log \varphi(\vx_i^T\inv{\ms}\vx_i).
\end{equation}
To compute a mle we equivalently consider the  minimisation problem~\eqref{eq.main2}, which we restate here for convenience
\begin{equation}
  \label{eq.main}
  \nlmin_{\ms \succ 0}\quad \Phi(\ms) := \tfrac12 n\log\det(\ms) - \nlsum_i \log\varphi(x_i^T\inv{\ms}x_i).
\end{equation}
Unfortunately, \eqref{eq.main} is in general very difficult: $\Phi$ may be nonconvex and may have multiple local minima (observe that $\log\det(S)$ is concave in $S$ and we are minimising). Since statistical estimation relies on having access to globally optimal estimates, it is important to be able to solve~\eqref{eq.main} globally. These difficulties notwithstanding, using our theory we identify a rich class of ECDs for which we can solve~\eqref{eq.main} globally. Some examples are already known~\citep{ollila11,kent91,zhang13}, but our techniques yield results strictly more general: they subsume previous examples while advancing the broader idea of geometric optimisation over HPD matrices.

Building on \S\ref{sec.gc} and \S\ref{sec.ln}, we divide our study into the following three classes of dgfs:
\begin{enumerate}[(i)]
  \setlength{\itemsep}{-1pt}
\item Geodesically convex (g-convex): This class contains functions for which the negative log-likelihood $\Phi(S)$ is g-convex. Some members of this class have been previously studied (though sometimes without recognising or directly exploiting g-convexity);
\item \textbf{L}og-\textbf{N}onexpansive (LN): This is a new class introduced in this paper. It exploits the ``non-positive curvature'' property of the HPD manifold. To our knowledge, this class of ECDs was beyond the grasp of previous methods~\citep{zhang13,kent91,wie12}. The iterative algorithm for finding the global minimum of the objective is similar to that of the class LC.
\item \textbf{L}og-\textbf{C}onvex (LC): We cover this class for completeness; it covers the case of log-convex $\varphi$, but leads to nonconvex $\Phi$ (due to the $-\log\varphi$ term). However, the structure of the problem is such that one can derive an efficient algorithm for finding a local minumum of the objective function.
\end{enumerate}
As illustrated in Figure~\ref{fig.fc}, these three classes can overlap. When a function is in the overlap between classes LC and g-convex, one can be sure that the iterative algorithm derived for the class LN will converge to a unique minimum. Table~\ref{tab:algo} summerizes the applicability of fixed-point or manifold optimization methods on different classes of dgfs.
\begin{figure}[h]
  \centering
  \includegraphics[scale=0.6]{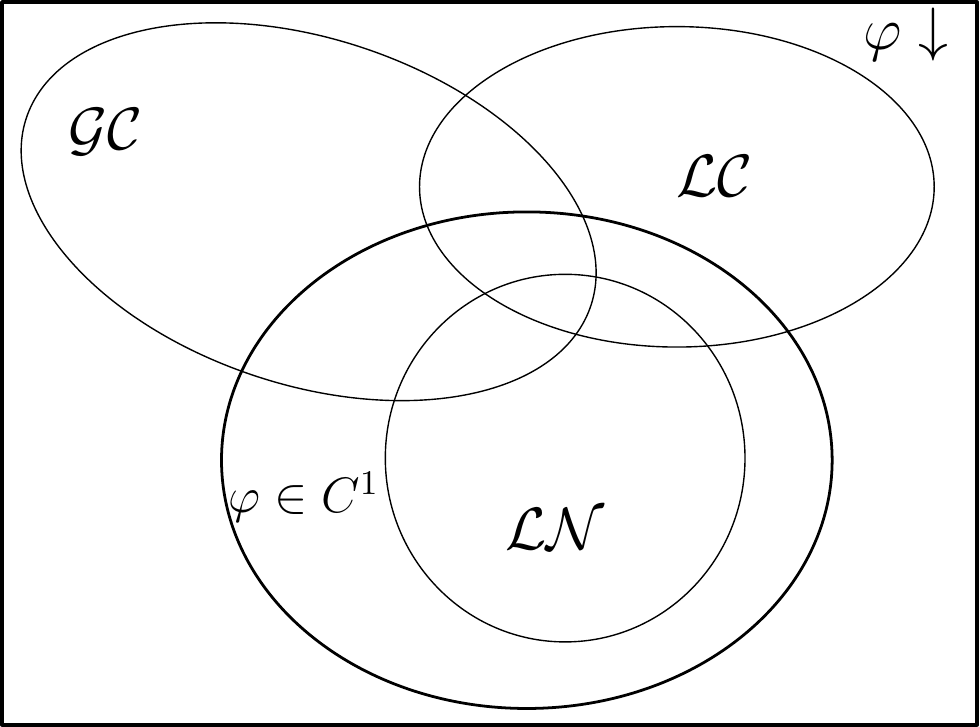}
  \caption{Overview of dgf functions classes for nonincreasing $\varphi$.}
  \label{fig.fc}
\end{figure}

\begin{table}[htpb]\small
  \centering
  \begin{tabular}{c|c|c}
    Problem Class & Manifold Opt. & Fixed-Point \\
    \hline
    $\cal GC$  & \textbf{Yes} & Can$^\star$\\
    $\cal LN$  & Can & \textbf{Yes} \\
    $\cal LC$  & Can & \textbf{Yes} 
  \end{tabular}
  \caption{\small Applicability of the different algorithms: `\textbf{Yes}' means a preferred algorithm; `Can$^\star$' denotes applicability on a case-by-case basis; `Can' signifies possible applicability of method.}
  \label{tab:algo}
\end{table}

\subsection{MLE for distributions in class g-convex}
\label{sec.ecd.gc}
If the log-likelihood is strictly g-convex then~\eqref{eq.main} cannot have multiple solutions. Moreover, for any local optimisation method that ensures a local solution to~\eqref{eq.main}, g-convexity ensures that this solution is globally optimal. 

First we state a corollary of Theorem~\ref{thm.gc} that helps recognise g-convexity of ECDs. We remark that a result equivalent to Corollary~\ref{cor.ecd} was also recently discovered in~\citep{zhang13}. Theorem~\ref{thm.gc} is more general and uses a completely different argument founded on matrix-theoretic results. %, so it may be of wider independent interest.
\begin{corr}
  \label{cor.ecd}
  Let $h: \reals_{++} \to \reals$ be g-convex (i.e., $h(x^{1-\lambda} y^{\lambda}) \le (1-\lambda) h(x)+\lambda h(y)$). If $h$ is nondecreasing, then for $r\in\set{\pm1}$, $\phi: \pp_d \to \reals : S \mapsto \nlsum_i h(x_i^TS^rx_i) \pm\log\det(S)$ is g-convex. Furthermore if $h$ is strictly g-convex, then $\phi$ is also strictly g-convex.
\end{corr}
\begin{proof}
  Immediate from Theorem~\ref{thm.gc} since $x_i^TS^rx_i$ is a positive linear map.
  %   \begin{equation*}
  %   h\bigl(\sqrt{\smash[b]{(x_i^T\ms^rx_i)(x_i^T\mr^rx_i)}}\bigr) \le \half h(x_i^T\ms^rx_i) + \half h(x_i^T\mr^rx_i).
  % \end{equation*}
\end{proof}

For reference, we summarise several examples of strictly g-convex ECDs in Corollary~\ref{corollary.kotz}.
\begin{corr}
  \label{corollary.kotz}
  The negative log-likelihood~\eqref{eq.main} is strictly g-convex for the following distributions:
  \begin{inparaenum}[(i)]
 % \item Gaussian;
  \item Kotz with $\alpha \leq \tfrac{d}{2}$ (its special cases include Gaussian, multivariate power exponential, multivariate W-distribution with shape parameter smaller than one, elliptical gamma with shape parameter $\nu \leq \tfrac{d}{2}$);
  \item Multivariate-$t$;
  \item Multivariate Pearson type II with positive shape parameter;
  \item Elliptical multivariate logistic distribution.
  \end{inparaenum}\footnote{The dgfs of different distributions are brought here for the reader's convenience. Multivariate power exponential: $\phi(t) = \exp(-t^\nu/b),\quad \nu>0$; Multivariate W-distribution: $\phi(t) = t^{\nu-1}\exp(-t^\nu/b),\quad \nu>0$; Elliptical gamma: $\phi(t) = t^{\nu-d/2}\exp(-t/b),\quad \nu>0$; Multivariate t: $\phi(t) = (1+t/\nu)^{-(\nu+d)/2},\quad \nu>0$; Multivariate Pearson type II: $\phi(t) = (1-t)^{\nu},\quad \nu>-1, 0\leq t\leq 1$; Elliptical multivariate logistic: $\phi(t) = \exp({-\sqrt{t}})/(1+\exp({-\sqrt{t}}))^2$.
  }
\end{corr}

Even though g-convexity ensures that every local solution will be globally optimal, we must first ensure that there exists a solution at all, that is,  \emph{does~\eqref{eq.main} have a solution?} Answering this question is nontrivial even in special cases~\citep{kent91,zhang13}. We provide below a fairly general result that helps establish existence. 

\begin{theorem}
\label{thm.existMin}
Let $\Phi(S)$ satisfy the following: (i) $-\log\varphi(t)$ is lower semi-continuous (lsc) for $t>0$, and (ii) $\Phi(S) \to \infty$ as $\norm{S} \to \infty$ or $\norm{S^{-1}} \to \infty$, then $\Phi(S)$ attains its minimum.
\end{theorem}
\begin{proof}
Consider the metric space $(\pp_d,d_R)$, where $d_R$ is the Riemannian distance,
\begin{equation}
  \label{eq.metr}
  d_R(A,B)= \frob{\log (A^{-1/2} B  A^{-1/2})}\qquad A, B \in \pp_d.
\end{equation}
If $\Phi(S) \to \infty$ as $\norm{S} \to \infty$ or as $\norm{S^{-1}} \to \infty$, then $\Phi(S)$ has bounded lower-level sets in $(\pp_d,d_R)$. It is a well-known result in variational analysis that an lsc function which has bounded lower-level sets in a metric space attains its minimum~\citep{rockwets}. Since  $-\log\varphi(t)$ is lsc and  $\log\det(S^{-1})$ is continuous, $\Phi(S)$ is lsc on  $(\pp_d,d_R)$. Therefore it attains its minimum.
\end{proof}

A key consequence of this theorem is its utility is in showing existence of solutions to~\eqref{eq.main} for a variety of different ECDs. We show an example application to Kotz-type distributions~\citep{kotz,kotz67} below.
% As an example application of the above ideas, we state an existence theorem for Kotz-type distributions.
For these distributions, the function $\Phi(\ms)$ assumes the form
\begin{equation}
\label{eq.kotzlog}
K(\ms)= \tfrac{n}{2}\log\det(\ms) + (\tfrac d 2- \alpha) \nlsum_{i=1}^n\log \vx_i^T\inv{\ms}\vx_i +  \nlsum_{i=1}^n \left ( \tfrac{\vx_i^T\inv{\ms}\vx_i}{b}\right )^\beta.
\end{equation}
Lemma~\ref{lem.Kotzinfty} shows that $K(\ms) \to \infty $ whenever $\norm{\ms^{-1}} \to \infty$ or $\norm {\ms} \to \infty$. 
\begin{lemma}
\label{lem.Kotzinfty}
Let the data $\Xc=\set{x_1,\ldots,x_n}$ span the whole space and for $\alpha < \tfrac d2$ satisfy
\begin{equation}
\label{eq.condF}
\frac{|\Xc \cap L|}{|\Xc|}<\frac{d_L}{d-2\alpha},
\end{equation}
where $L$ is an arbitrary subspace with dimension $d_L<d$ and $|\Xc \cap L|$ is the number of datapoints that lie in the subspace $L$.
If $\norm{\ms^{-1}} \to \infty$ or $\norm {\ms} \to \infty$, then $K(\ms) \to \infty $.
\end{lemma}
\begin{proof}
If $\norm{\ms^{-1}} \to \infty$ and since the data span the whole space, it is possible to find a datum $x_1$ such that $t_1=\vx_1^T\inv{\ms}\vx_1 \to \infty$. Since 
\begin{equation*}
\lim\limits_{t \to \infty} c_1 \log(t) + t^{c_2}+c_3 \to \infty
\end{equation*}
for constants $c_1$,$c_3$ and $c_2 >0$, it follows that  $K(\ms) \to \infty $ whenever $\norm{\ms^{-1}} \to \infty$. 

If $\norm{\ms} \to \infty$ and $\norm{\ms^{-1}}$ is bounded, then the third term in expression of $K(\ms)$ is bounded. Assume that $d_L$ is the number of eigenvalues of $\ms$ that go to $\infty$ and $|\Xc \cap L|$ is the number of data that lie in the subspace span by these eigenvalues. Then in the limit when eigenvalues of $\ms$ go to $\infty$, $K(\ms)$ converges to the following limit 
\begin{equation*}
\lim\limits_{\lambda \to \infty} \tfrac{n}{2}d_L\log \lambda + (\tfrac d 2- \alpha) |\Xc \cap L| \log \lambda^{-1} +c
\end{equation*}
Apparently if $\tfrac{n}{2}d_L + (\tfrac d 2- \alpha) |\Xc \cap L| >0$, then $K(\ms) \to \infty$ and the proof is complete.
\end{proof}
% \begin{proof}
%   Please see supplement.
% \end{proof}
It is important to note that overlap condition~\eqref{eq.condF} can be fulfilled easily by assuming that the number of data points is larger than their dimensionality and that they are noisy. Using Lemma~\ref{lem.Kotzinfty}  with Theorem~\ref{thm.existMin} we obtain the following key result for Kotz-type distributions. %show the cost function of the Kotz-type distribution attains its minimum.
\begin{theorem}[MLE existence]
  \label{thm.Kotzexist}
  If the data samples satisfy condition~\eqref{eq.condF}, then log-likelihood of Kotz-type distribution has a maximiser (i.e., there exists an mle).
  %\commR{This condition may be relaxed based on the existence results in\cite{kent91}.}
 % \note{S: In a NIPS paper, good to not mention / rely too much on these papers;}
\end{theorem}

\subsubsection{Optimisation algorithm} 
Once existence is ensured, one may use any local optimisation method to minimise~\eqref{eq.main} to obtain the desired mle. For members of the class g-convex that do not lie in class LN or class LC, we recommend  invoking the manifold optimisation techniques summarised in \S\ref{sec.manopt}.

\subsection{MLE for distributions in class LN}
\label{sec.mle.ln}
For negative log-likelihoods~\eqref{eq.main} in class LN, we can circumvent the heavy machinery of manifold optimisation, and obtain simple fixed-point algorithms by appealing to the contraction results developed in \S\ref{sec.ln}. We note that some members of class g-convex may also turn out to lie in class LN, so the discussion below also applies to them. 

As an illustrative example of these results, consider the problem of finding the minimum of negative log-likelihood solution of Kotz-type distribution~\eqref{eq.kotzlog}. If the corresponding nonlinear equation~\eqref{eq.fixediter} with corresponding $h(.)=(\tfrac d2 -\alpha)(.)^{-1}+\tfrac{\beta}{b^{\beta}} (.)^{\beta-1}$ has a positive definite solution, then it is a candidate mle; if it is unique, then it is the desired solution to~\eqref{eq.kotzlog}.

But how should we  solve~\eqref{eq.fixediter}? This is where the theory developed in \S\ref{sec.ln} comes into play. Convergence of the iteration~\eqref{eq:iterform} as applied to~\eqref{eq.fixediter} can be obtained from Theorem~\ref{thm.thomln}. But in the Kotz case we can actually show a stronger result that helps ensure better geometric convergence rates for the fixed-point iteration.

\begin{lemma}
\label{lem.lnKotz}
If $c\geq0$ and $-1<\tau<1$, then $g(x)=cx + x^\tau$ is log-contractive.
\end{lemma}
\begin{proof}
Without loss of generality assume $t=ks$ with $k \geq 1$. Assume that $g(t) \geq g(s)$:
 \begin{align*}
  \log g(t) \quad&=\quad \log (ct+t^\tau)\\
 \quad&=\quad \log (kcs+k^\tau s^\tau ) \\
 \quad&=\quad \log (k(cs+s^\tau )+k^\tau s^\tau -ks^\tau ) \\
 \quad&=\quad \log k(cs+s^\tau ) \Bigl( 1+ \frac{k^\tau s^\tau -ks^\tau }{k(cs+s^\tau )} \Bigr) \\
 \quad&=\quad \log k + \log g(s) + \log \Bigl ( 1+ \frac{s^\tau (k^{\tau -1}-1)}{(cs+s^\tau )} \Bigr) \\
 |\log g(t) -\log g(s)|\quad&=\quad|\log t- \log s| + \log \Bigl ( 1+ \frac{s^\tau (k^{\tau -1}-1)}{(cs+s^\tau )} \Bigr )
 \end{align*}
 Since the second term is negative, therefore $g$ is log-contractive. Consider the other case $g(t) \geq g(s)$, that could happen only when $\tau  \leq 0$.
 \begin{align*}
   \log g(s) \quad&=\quad \log (cs+s^\tau )\\
  \quad&=\quad \log (ct/k+k^{|\tau |}t^\tau ) \\
  \quad&=\quad \log (k(ct+t^\tau )+k^{|\tau |}t^\tau +ct/k-ckt-kt^\tau ) \\
  \quad&=\quad \log k(ct+t^\tau ) \Bigl( 1+ \frac{k^{|\tau |}t^\tau+ ct/k-ckt-kt^\tau }{k(ct+t^\tau )} \Bigr) \\
  \quad&=\quad \log k + \log g(t) + \log \Bigl( 1+ \frac{ct\left (k^{-2}-1 \right)+t^\tau (k^{|\tau |-1}-1)}{(ct+t^\tau )} \Bigr)\\
  |\log g(t) -\log g(s)|\quad&=\quad|\log t- \log s| + \log \Bigl ( 1+ \frac{ct\left (\frac{1}{k^2}-1 \right )+t^\tau (k^{|\tau |-1}-1)}{(ct+t^\tau )} \Bigr).
 \end{align*}
 In this case, the second term is also negative. Therefore $h$ is log-contractive.
\end{proof}

Assume $\tau =\beta-1$, $c=\tfrac{b^{\beta}(d/2 -\alpha)}{\beta}$ and knowing that $h(.)=g(\beta b^{-\beta} (.))$ has the same contraction factor as $g(.)$, Lemma~\ref{lem.lnKotz} implies that $h$ in the iteration~\eqref{eq.fixediter} for the Kotz-type distributions with $0<\beta<2$ and $\alpha\leq\tfrac d 2$ is log-contractive. Based on Theorem~\ref{thm.Kotzexist}, $K(\ms)$ has at least one minimum. Thus using Theorem~\ref{thm.ln.convg}, we have the following main convergence result.
\begin{theorem}
  \label{thm.kotz.cvg}
If the data samples satisfy~\eqref{eq.condF}, then  iteration~\eqref{eq.fixediter} for Kotz-type distributions with $0<\beta<2$ and $\alpha\leq\tfrac d 2$ converges to a unique fixed point.
\end{theorem}

\subsection{MLE for distributions in class LC}
\label{sec.lc}
For completeness, we briefly mention class LC here, which is perhaps one of the most studied classes of ECDs, at least from an algorithmic point-of-view~\citep{kent91}. Therefore, we only discuss it summarily, and present our new results. 

For the class LC, we assume that the dgf $\varphi$ is log-convex. Without assumptions that are typically made in the literature, it can be that neither the GC nor the LN analysis applies to class LC. However, the optimisation problem still has structure that allows simple and efficient algorithms. Specifically, here the objective function $\Phi(\ms)$ can be written as a a difference of two convex functions by introducing the variable $P=\inv{\ms}$, wherewith we have $\Phi(P) = -an\log\det(P) - \nlsum_i\log\varphi(\vx_i^TP\vx_i)$. %The first term is convex in $P$, while the second is concave (since $-\log\varphi$ is concave). %, so its composition with an affine function of $P$, i.e., $\vx_i^TP\vx_i = \trace(P\vx_i\vx_i^T)$ is also concave).

To this representation of $\Phi$ we may now apply the CCCP procedure~\citep{cccp} to search for a locally optimal point. The method operates as follows
\begin{equation*}
  P^{k+1} \gets \argmin_{P \succ 0}\quad -\tfrac n 2\log\det(P) + \trace\left(P\nlsum_ih(\vx_i^TP^k\vx_i)\vx_i\vx_i^T\right),
\end{equation*}
which yields the update
\begin{equation}
  \label{eq.7}
  P^{k+1} = \left(\tfrac{2}{n}\nlsum_ih(\vx_i^TP^k\vx_i)\vx_i\vx_i^T\right)^{-1}.
\end{equation}

Because $P^{k+1}$ is constructed using the CCCP procedure, it can be shown that the sequence $\set{\Phi(P^k)}$ is monotonically decreasing. Furthermore, since we assumed $h$ to be nonnegative, therefore the iteration stays within positive semidefinite cone. If the cost function goes to infinity whenever the covariance matrix is singular, then using Theorem~\ref{thm.existMin} we can conclude that iteration converges to a positive definite matrix. Thus, we can state the following key result for class LC.
\begin{theorem}[Convergence]
\label{thm.lcconverge}
Assume that $\Phi(P)$ goes to infinity whenever $P$ reaches the boundary of $\pp^d$, i.e. 
$\norm{P} \to \infty  \lor \norm{P^{-1}} \to \infty \implies \Phi(P) \to \infty$. Furthermore if $-\log\varphi$ is concave and $h$ is non-negative, then each step of the iterative algorithm given in~\eqref{eq.7} decreases the cost function and furthermore it converges to a positive definite solution.
\end{theorem}

A similar theorem but under more strict conditions was established in \citet{kent91}. Knowing that the iterative algorithm in~\eqref{eq.7} is the same as   \eqref{eq.fixediter} and using Theorem~\ref{thm.lcconverge} with the existence result of Theorem~\ref{thm.Kotzexist} and the uniqueness result of Corollary~\ref{corollary.kotz}, we can state the following theorem for Kotz-type distributions (\emph{cf.} Theorem~\ref{thm.kotz.cvg}).
\begin{theorem}
  \label{thm.kotz.cvg2}
  If the data samples satisfy condition~\eqref{eq.condF}, then  iteration~\eqref{eq.fixediter} for Kotz-type distributions with $\beta\geq1$ and $\alpha\leq\tfrac d 2$ converges to a unique fixed point.
\end{theorem}

Theorem~\ref{thm.kotz.cvg2} and Theorem~\ref{thm.kotz.cvg} together show that the iteration~\eqref{eq.fixediter} for Kotz-type distributions with $\alpha\leq\tfrac d 2$ and regardless of the value of $\beta$ always converges to the unique mle estimate whenever it exists.
%The author also gave conditions under which $\Phi(P)$ becomes infinity when the matrix becomes singular. \citet{kent91} also provide conditions under which the iterative algorithm given in~\eqref{eq.7} converges to a unique solution. The result of previous section vastly generalized the uniqueness result of \citet{kent91} and also some of its extensions given in \citet{ollila11}.

\section{Numerical results}
\label{sec.expt}
We briefly illustrate the numerical performance of our fixed-point iteration. The key message here is that our fixed-point iterations solve nonconvex problems that are further complicated by a positive definiteness constraint. But by construction the fixed-point iterations satisfy the constraint, so no extra eigenvalue computation is needed, which can provide substantial computational savings. In contrast, a general nonlinear solver must perform constrained optimisation, which may be unduly expensive.

We show two experiments (Figs.~\ref{fig.dim} and \ref{fig.two}) to demonstrate scalability of the fixed-point iteration with increasing dimensionality of the input matrix and for varying $\beta$ parameter of the Kotz distribution which influences convergence rate of our fixed-point iteration. For all simulations, we sampled 10,000 datapoints from the Kotz-type distribution with given $\alpha$ and $\beta$ parameters and a random covariance matrix.

%For three different dimensions $d=4$, $d=16$, and $d=32$, we sample 10,000 datapoints from a Kotz-type distribution with $\beta=0.5$ and $\alpha=2$ and a random covariance matrix. The convergence speed is shown as red curves in Figure~\ref{fig.dim}. 

We note that the problems are nonconvex with an open set as a constraint---this precludes direct application of semidefinite programming or  approaches such as gradient-projection (projection requires closed sets). We also tried interior-point methods but we did not include them in the comparisons because of their extremely slow convergence speed on this problem. So we choose to show the result of (Riemannian) manifold optimisation techniques~\citep{absil}.

We compare our fixed-point iteration against four different manifold optimisation methods: (i) steepest descent (SD); (ii) conjugate gradients (CG); (iii) trust-region (TR); and (iv) LBFGS, which implements Algorithm~\ref{alg.lrbfgs}. All methods are implemented in \textsc{Matlab} (including the fixed-point iteration); for manifold optimisation we extend the \textsc{Manopt} toolbox~\citep{manopt} to support the HPD manifold\footnote{The newest version of the \textsc{Manopt} toolbox ships with an implementation of the HPD manifold, but we use our own implementation as it includes some utilities specific to LBFGS.} as well as Algorithm~\ref{alg.lrbfgs}.

  % of constrained optimisation (blue curves) using MATLAB's optimisation toolbox are shown. % We used the procedure used in \citep{vanderbei2000} that converts semi positive definite constraints to a set of concave nonlinear constraints.
From Figure~\ref{fig.dim} we see that the basic fixed-point algorithm (FP) does not perform better than SD, the simplest manifold optimisation method.  Moreover, even when FP performs better than CG, TR, or LBFGS (Figure~\ref{fig.two}), it seems to closely follow SD. However, the scaling idea introduced in \S\ref{sec.example} leads to a fixed-point method (FP2) that outperforms all other methods, both with increasing dimensionality and varying $\beta$. The scale is chosen by ensuring $\text{trace}(\mathcal{M}(S_{k+1})) = d$.

These results merely indicate that the fixed-point approach can be  competitive. A more thorough experimental study to assess our algorithms remains to be undertaken.

%\iffalse
\begin{figure*}[htbp]
\centering
\begin{subfigure}[b]{0.3\textwidth}
\includegraphics[width=\textwidth]{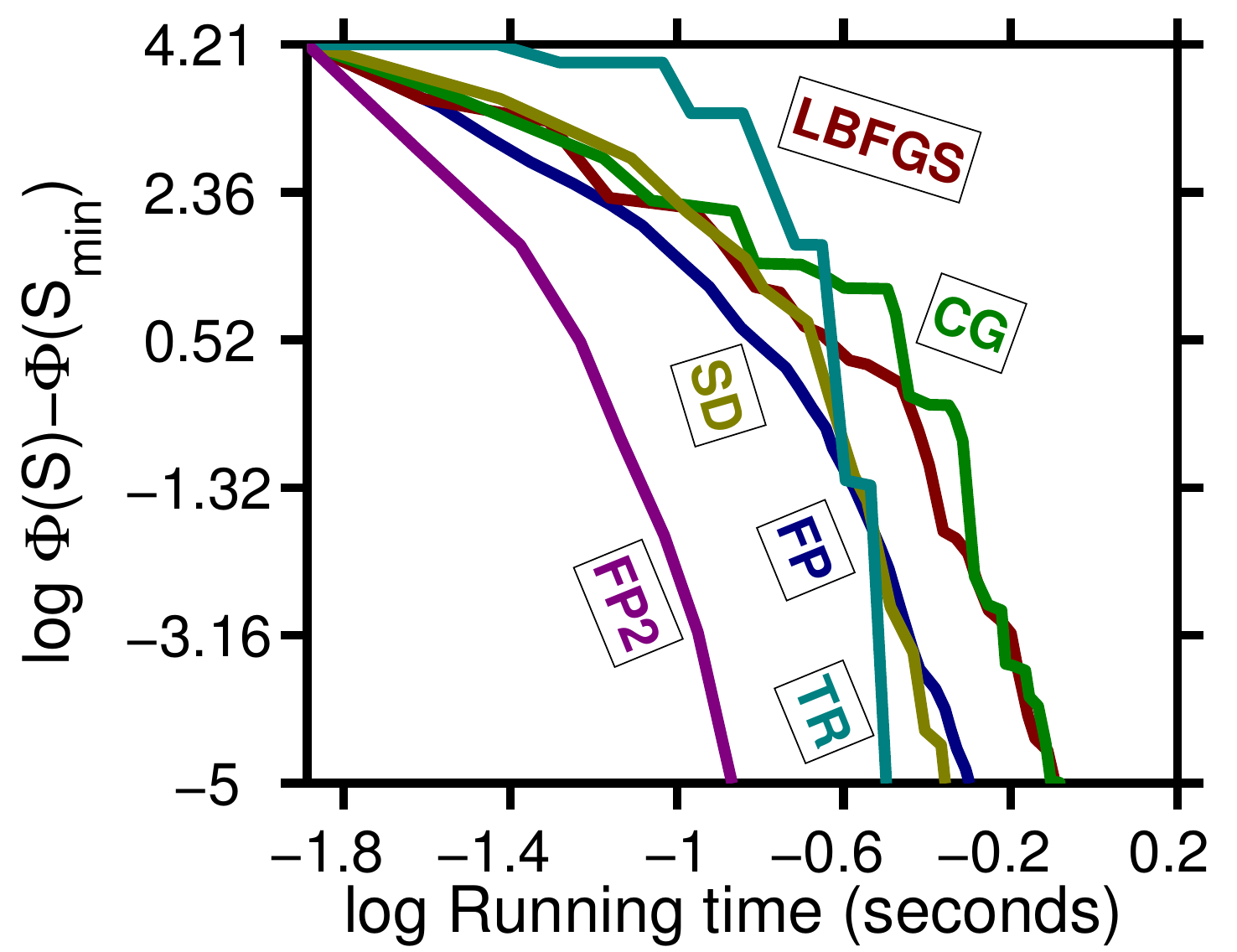} 
%\vspace {0.1\textwidth}
%\caption{\label{fig-dim-a}bla}
\end{subfigure}
\hspace {0.02\textwidth}
\begin{subfigure}[b]{0.3\textwidth}
\includegraphics[width=\textwidth]{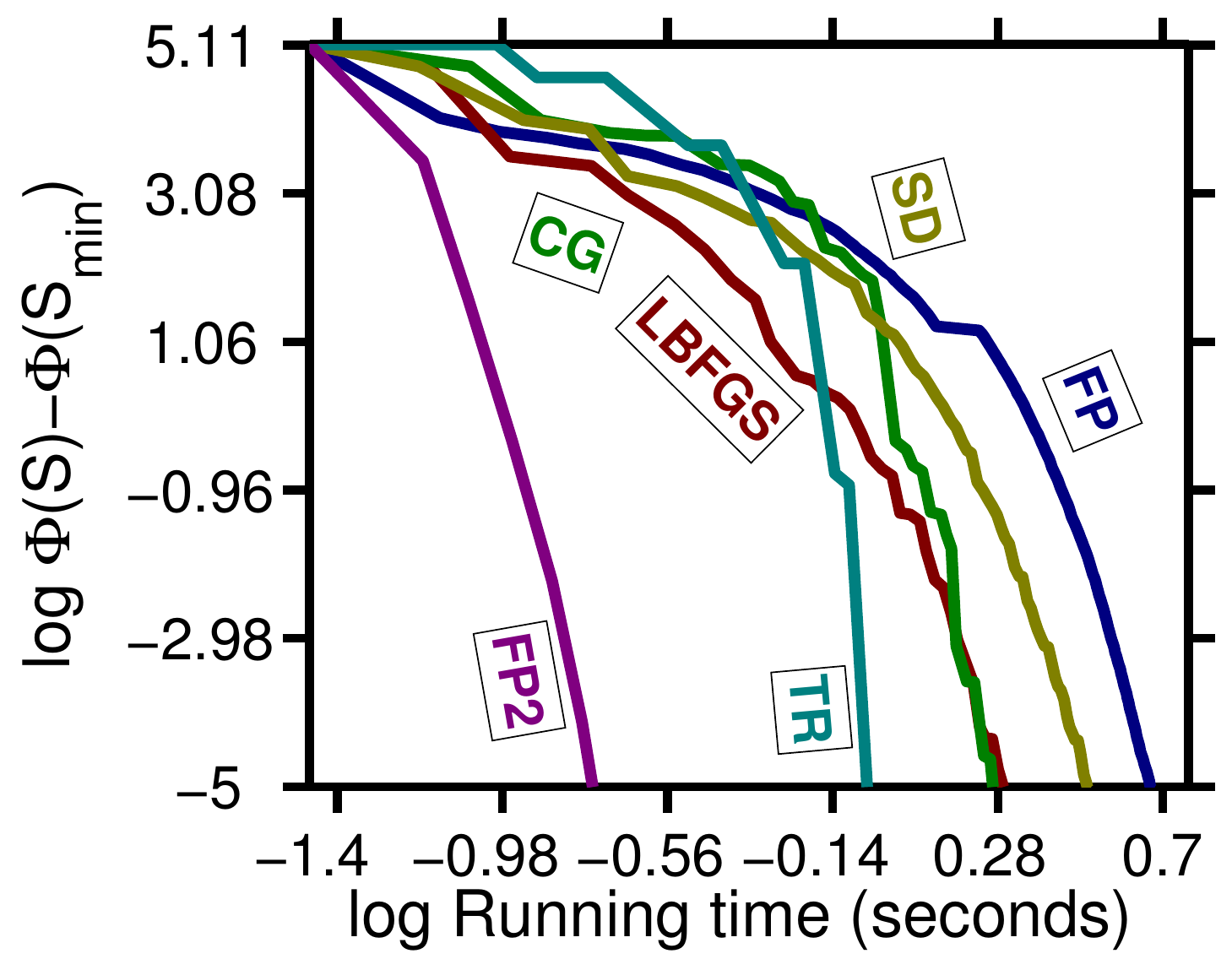}
%\vspace {0.1\textwidth}
%\caption{\label{fig-dim-b}bla}
\end{subfigure}
\begin{subfigure}[b]{0.3\textwidth}
\includegraphics[width=\textwidth]{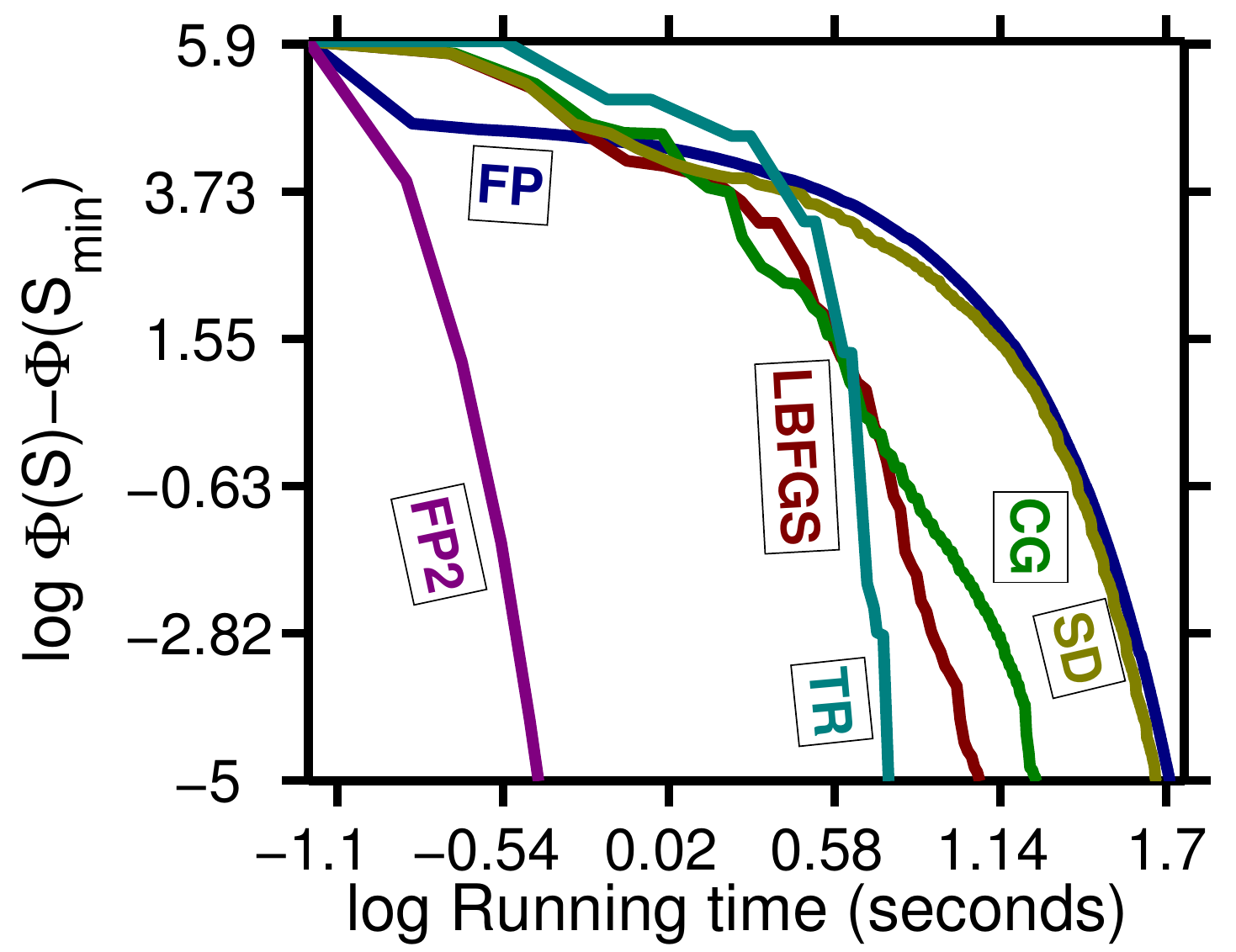} 
%\vspace {0.1\textwidth}
%\caption{\label{fig-dim-a}bla}
\end{subfigure}
\caption{\small Running times comparison of the fixed-point iterations compared with four different manifold optimization techniques to maximise a Kotz-likelihood with $\beta=0.5$ and $\alpha=1$ (see text for details). FP denoted normal fixed-point iteration and FP2 is the fixed-point iteration with scaling factor. Manifold optimization methods are steepest descent (SD), conjugate gradient (CG), limited-memory RBFGS (LBFGS) and trust-region (TR) . The plots show (from left to right), running times for estimating $S \in \pp_d$, for $d\in\set{4,16,64}$.}
\label{fig.dim}
\end{figure*}

\begin{figure*}[h]
  \centering
  \includegraphics[width=.3\textwidth]{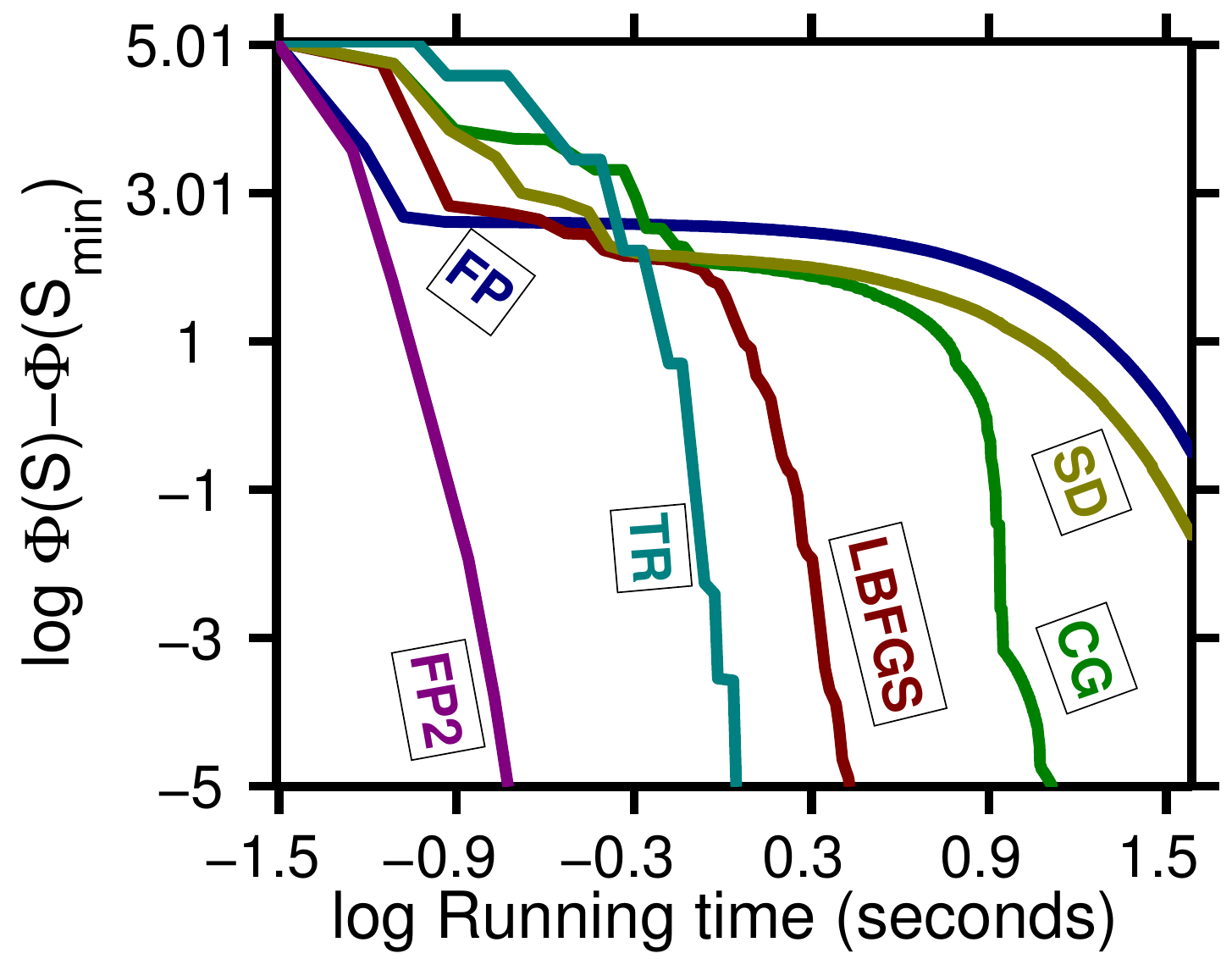}
  \includegraphics[width=.3\textwidth]{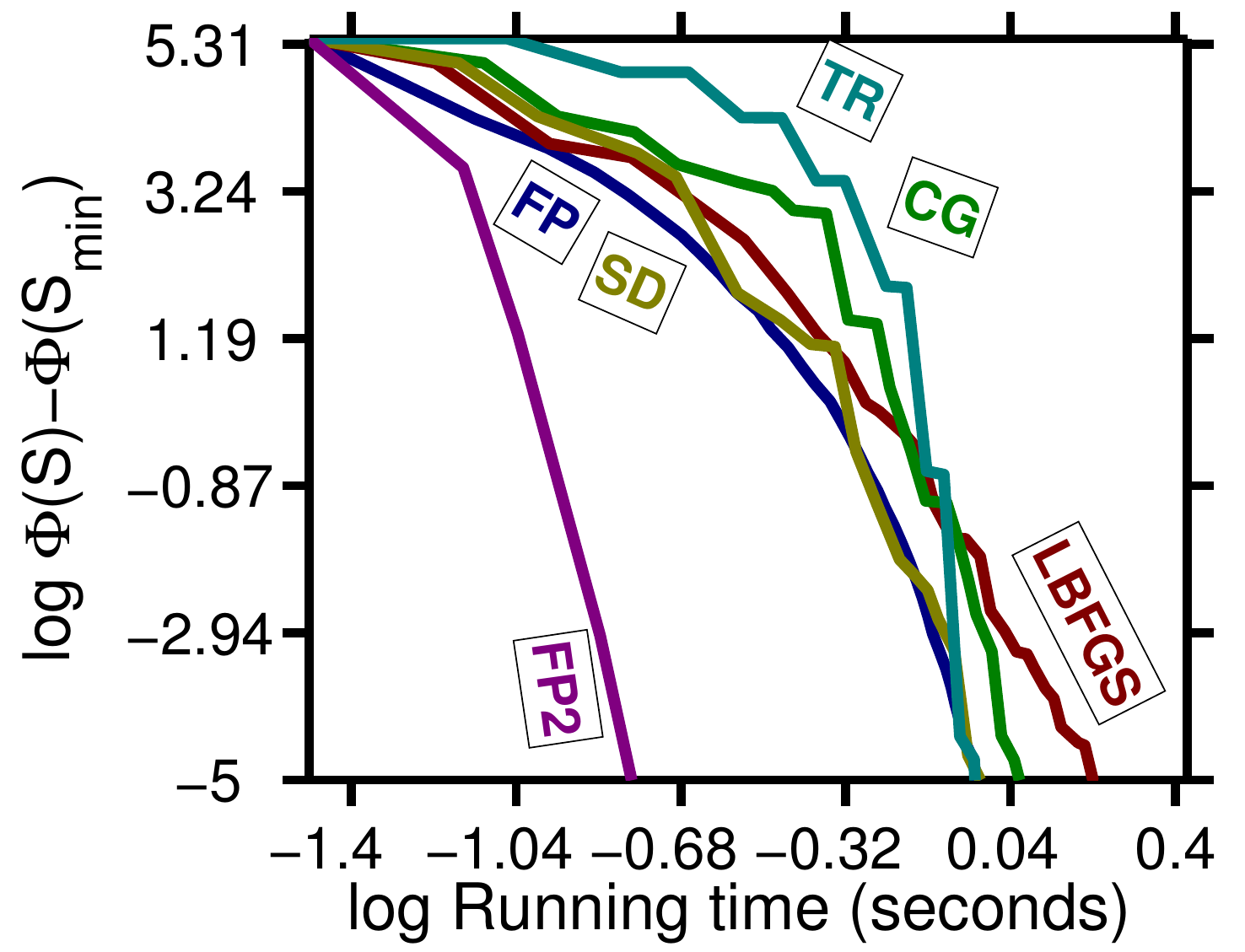}
  \includegraphics[width=.3\textwidth]{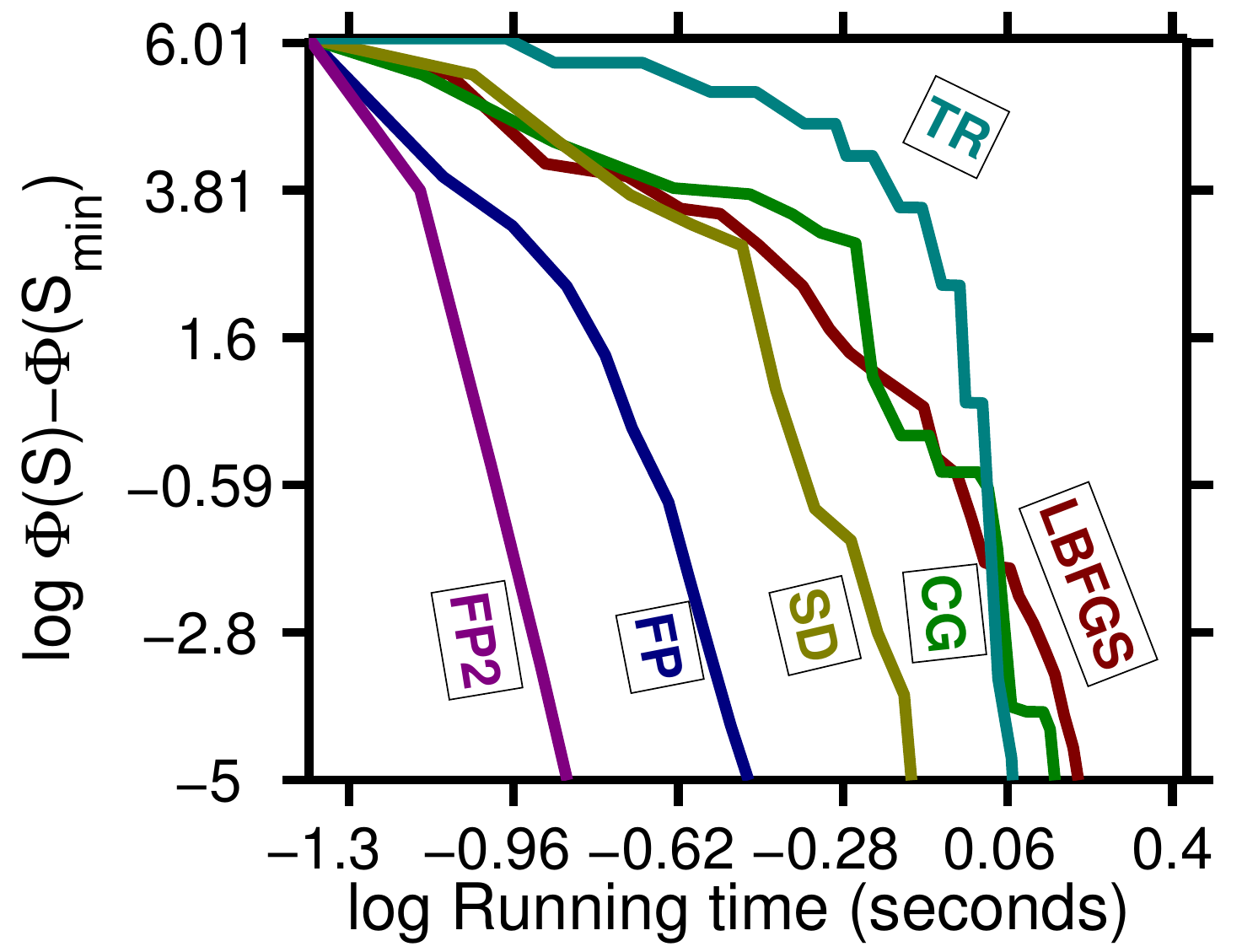}
  \caption{\small In the Kotz-type distribution, when $\beta$ gets close to zero or 2 or when $\alpha$ gets close to zero, the contraction factor becomes smaller which can impact the convergence rate. This figure shows running time variance for Kotz-type distributions with $d=16$ and $\alpha=2 \beta$ for different values of $\beta \in \{0.1, 1, 1.7\}$.}\label{fig.two}
\end{figure*}
%\fi

%\begin{table}[htpb]\small
%  \centering
%  \begin{tabular}{c|c|c|c}
%    Problem Class & Manifold Opt. & Fixed-Point & CCCP\\
%    \hline
%    $\cal GC$  & \textbf{Yes} & Can$^\star$ & Can\\
%    $\cal LN$  & Can & \textbf{Yes} & Can\\
%    $\cal LC$  & Can & Can$^\star$ & \textbf{Yes}
%  \end{tabular}
%  \caption{\small Applicability of the different algorithms: `\textbf{Yes}' means a preferred algorithm; `Can$^\star$' denotes applicability on a case-by-case basis; `Can' signifies possible applicability of method.}
%  \label{tab:algo}
%\end{table}

\section{Conclusion}
We studied geometric optimisation for minimising certain nonconvex functions over the set of positive definite matrices. We showed key results that help recognise geodesic convexity; we also introduced a new class of log-nonexpansive functions which contains functions that need not be geodesically convex, but can still be optimised efficiently. Key to our ideas was a construction of fixed-point iterations in a suitable metric space on positive definite matrices. 

Additionally, we developed and applied our results in the context of maximum likelihood estimation for elliptically contoured distributions, covering instances substantially beyond the state-of-the-art. We believe that the general geometric optimisation techniques that we developed in this paper will prove to be of wider use and interest beyond our motivating examples and applications. Moreover, developing a more extensive geometric optimisation numerical package is an ongoing project.

\bibliographystyle{plainnat}
\setlength{\bibsep}{2pt}
%\bibliography{../ecd}

\begin{thebibliography}{52}
\providecommand{\natexlab}[1]{#1}
\providecommand{\url}[1]{\texttt{#1}}
\expandafter\ifx\csname urlstyle\endcsname\relax
  \providecommand{\doi}[1]{doi: #1}\else
  \providecommand{\doi}{doi: \begingroup \urlstyle{rm}\Url}\fi

\bibitem[Absil et~al.(2009)Absil, Mahony, and Sepulchre]{absil}
P.-A. Absil, R.~Mahony, and R.~Sepulchre.
\newblock \emph{Optimization algorithms on matrix manifolds}.
\newblock Princeton University Press, 2009.

\bibitem[Ando(1979)]{ando79}
T.~Ando.
\newblock Concavity of certain maps of positive definite matrices and
  applications to {H}adamard products.
\newblock \emph{Linear Algebra and its Applications}, 26:\penalty0 203--221,
  1979.

\bibitem[Ando et~al.(2004)Ando, Li, and Mathias]{andoLiMa04}
T.~Ando, C.-K. Li, and R.~Mathias.
\newblock Geometric means.
\newblock \emph{Linear Algebra and its Applications}, 385:\penalty0 305--334,
  2004.

\bibitem[Arnaudon et~al.(2013)Arnaudon, Barbaresco, and Yang]{barber}
M.~Arnaudon, F.~Barbaresco, and L.~Yang.
\newblock Riemannian medians and means with applications to radar signal
  processing.
\newblock \emph{{IEEE} Journal on Selected Topics in Signal Processing},
  7\penalty0 (4):\penalty0 595--604, 2013.

\bibitem[Bhatia(1997)]{bhatia97}
R.~Bhatia.
\newblock \emph{Matrix Analysis}.
\newblock Springer, 1997.

\bibitem[Bhatia(2007)]{bhatia07}
R.~Bhatia.
\newblock \emph{Positive Definite Matrices}.
\newblock Princeton University Press, 2007.

\bibitem[Bhatia and Karandikar(2011)]{bhatia11}
R.~Bhatia and R.~L. Karandikar.
\newblock The matrix geometric mean.
\newblock Technical Report isid/ms/2-11/02, Indian Statistical Institute, 2011.

\bibitem[Bini and Iannazzo(2013)]{bruno}
D.~A. Bini and B.~Iannazzo.
\newblock Computing the {K}archer mean of symmetric positive definite matrices.
\newblock \emph{Linear Algebra and its Applications}, 438\penalty0
  (4):\penalty0 1700--1710, 2013.

\bibitem[Blekherman and Parrilo(2013)]{parrilo}
G.~Blekherman and P.~A. Parrilo, editors.
\newblock \emph{Semidefinite Optimization and Convex Algebraic Geometry}.
\newblock SIAM, 2013.

\bibitem[Boumal et~al.(2014)Boumal, Mishra, Absil, and Sepulchre]{manopt}
Nicolas Boumal, Bamdev Mishra, P.-A. Absil, and Rodolphe Sepulchre.
\newblock {M}anopt, a {M}atlab toolbox for optimization on manifolds.
\newblock \emph{Journal of Machine Learning Research}, 15:\penalty0 1455--1459,
  2014.
\newblock URL \url{http://www.manopt.org}.

\bibitem[Boyd et~al.(2007)Boyd, Kim, Vandenberghe, and Hassibi]{boyd.gp}
S.~Boyd, S.-J. Kim, L.~Vandenberghe, and A.~Hassibi.
\newblock A tutorial on geometric programming.
\newblock \emph{Optimization and Engineering}, 8\penalty0 (1):\penalty0
  67--127, 2007.

\bibitem[Bridson and Haeflinger(1999)]{bridson}
M.~R. Bridson and A.~Haeflinger.
\newblock \emph{Metric Spaces of Non-Positive Curvature}.
\newblock Springer, 1999.

\bibitem[Cambanis et~al.(1981)Cambanis, Huang, and Simons]{cambanis81}
S.~Cambanis, S.~Huang, and G.~Simons.
\newblock On the theory of elliptically contoured distributions.
\newblock \emph{Journal of Multivariate Analysis}, 11\penalty0 (3):\penalty0
  368--385, 1981.

\bibitem[Chebbi and Moahker(2012)]{chebbi}
Z.~Chebbi and M.~Moahker.
\newblock Means of {H}ermitian positive-definite matrices based on the
  log-determinant $\alpha$-divergence function.
\newblock \emph{Linear Algebra and its Applications}, 436:\penalty0 1872--1889,
  2012.

\bibitem[Chen et~al.(2011)Chen, Wiesel, and Hero]{chen11}
Y.~Chen, A.~Wiesel, and A.~O. Hero.
\newblock Robust shrinkage estimation of high-dimensional covariance matrices.
\newblock \emph{{IEEE} Transactions on Signal Processing}, 59\penalty0
  (9):\penalty0 4097--4107, 2011.

\bibitem[Cheng and Vemuri(2013)]{vemuri}
G.~Cheng and B.~Vemuri.
\newblock A novel dynamic system in the space of spd matrices with applications
  to appearance tracking.
\newblock \emph{SIAM Journal on Imaging Sciences}, 6\penalty0 (1):\penalty0
  592--615, 2013.

\bibitem[Cheng et~al.(2012)Cheng, Salehian, and Vemuri]{vemuri12}
G.~Cheng, H.~Salehian, and B.~C. Vemuri.
\newblock Efficient recursive algorithms for computing the mean diffusion
  tensor and applications to {DTI} segmentation.
\newblock In \emph{European Conference on Computer Vision (ECCV)}, volume~7,
  pages 390--401, 2012.

\bibitem[Cherian et~al.(2012)Cherian, Sra, Banerjee, and
  Papanikolopoulos]{chSra12}
A.~Cherian, S.~Sra, A.~Banerjee, and N.~Papanikolopoulos.
\newblock Jensen-bregman logdet divergence for efficient similarity
  computations on positive definite tensors.
\newblock \emph{{IEEE} Transactions Pattern Analysis and Machine Intelligence},
  2012.

\bibitem[Choi(1975)]{choi75}
M.-D. Choi.
\newblock Completely positive linear maps on complex matrices.
\newblock \emph{Linear Algebra and its Applications}, 10:\penalty0 285--290,
  1975.

\bibitem[Edelstein(1962)]{edel62}
M.~Edelstein.
\newblock On fixed and periodic points under contractive mappings.
\newblock \emph{Journal of the London Mathematical Society}, s1-37\penalty0
  (1):\penalty0 74--79, 1962.

\bibitem[Gupta and Nagar(1999)]{gupta99}
A.~K. Gupta and D.~K. Nagar.
\newblock \emph{Matrix Variate Distributions}.
\newblock Chapman and Hall/CRC, 1999.

\bibitem[Gurvits and Samorodnitsky(2002)]{gurvits}
L.~Gurvits and A.~Samorodnitsky.
\newblock A deterministic algorithm for approximating mixed discriminant and
  mixed volume, and a combinatorial corollary.
\newblock \emph{Disc. Comp. Geom.}, 27\penalty0 (4), 2002.

\bibitem[Hardy et~al.(1929)Hardy, Littlewood, and P\'olya]{haLiPo}
G.~H. Hardy, J.~E. Littlewood, and G.~P\'olya.
\newblock Some simple inequalities satisfied by convex functions.
\newblock \emph{Messenger Math.}, 58:\penalty0 145--152, 1929.

\bibitem[Hiai and Petz(2012)]{hiai12}
F.~Hiai and D.~Petz.
\newblock Riemannian metrics on positive definite matrices related to means.
  {II}.
\newblock \emph{Linear Algebra and its Applications}, 436\penalty0
  (7):\penalty0 2117--2136, April 2012.

\bibitem[Jeuris et~al.(2012)Jeuris, Vandebril, and Vandereycken]{jeuVaVa}
B.~Jeuris, R.~Vandebril, and B.~Vandereycken.
\newblock A survey and comparison of contemporary algorithms for computing the
  matrix geometric mean.
\newblock \emph{Electronic Transactions on Numerical Analysis}, 39:\penalty0
  379--402, 2012.

\bibitem[K.-T.~Fang and Ng(1990)]{kotz}
S.~Kotz K.-T.~Fang and K.~W. Ng.
\newblock \emph{Symmetric multivariate and related distributions}.
\newblock Chapman and Hall, 1990.

\bibitem[Kent and Tyler(1991)]{kent91}
J.~T. Kent and D.~E. Tyler.
\newblock Redescending {M}-estimates of multivariate location and scatter.
\newblock \emph{The Annals of Statistics}, 19\penalty0 (4):\penalty0
  2102--2119, December 1991.

\bibitem[Kotz et~al.(1967)Kotz, Johnson, and Boyd]{kotz67}
S.~Kotz, N.~L. Johnson, and D.~W. Boyd.
\newblock Series representations of distributions of quadratic forms in normal
  variables. {I}. central case.
\newblock \emph{The Annals of Mathematical Statistics}, 38\penalty0
  (3):\penalty0 823--837, June 1967.

\bibitem[Kraus(1971)]{kraus71}
K.~Kraus.
\newblock General state changes in quantum theory.
\newblock \emph{Annals of Physics}, 64:\penalty0 311--335, 1971.

\bibitem[Kubo and Ando(1980)]{kuboAndo80}
F.~Kubo and T.~Ando.
\newblock Means of positive linear operators.
\newblock \emph{Mathematische Annalen}, 246:\penalty0 205--224, 1980.

\bibitem[Lee and Lim(2008)]{leeLim}
H.~Lee and Y.~Lim.
\newblock Invariant metrics, contractions and nonlinear matrix equations.
\newblock \emph{Nonlinearity}, 21:\penalty0 857--878, 2008.

\bibitem[Lemmens and Nussbaum(2012)]{lemNuss12}
B.~Lemmens and R.~Nussbaum.
\newblock \emph{Nonlinear Perron-Frobenius Theory}.
\newblock Cambridge University Press, 2012.

\bibitem[Lim and P\'alfia(2012)]{limPal12}
Y.~Lim and M.~P\'alfia.
\newblock Matrix power means and the {K}archer mean.
\newblock \emph{Journal of Functional Analysis}, 262:\penalty0 1498--1514,
  2012.

\bibitem[Loan(2000)]{vanloan00}
C.~F.~Van Loan.
\newblock The ubiquitous {K}ronecker product.
\newblock \emph{Journal of Computational and Applied Mathematics},
  123:\penalty0 85--100, 2000.

\bibitem[Matharu and Aujla(2012)]{maAu12}
J.~S. Matharu and J.~S. Aujla.
\newblock Some inequalities for unitarily invariant norms.
\newblock \emph{Linear Algebra and its Applications}, 436:\penalty0 1623--1631,
  2012.

\bibitem[Moakher(2005)]{moakher}
M.~Moakher.
\newblock A differential geometric approach to the geometric mean of symmetric
  positive-definite matrices.
\newblock \emph{SIAM Journal on Matrix Anal. Appl. (SIMAX)}, 26:\penalty0
  735--747, 2005.

\bibitem[Muirhead(1982)]{muirhead82}
R.~J. Muirhead.
\newblock \emph{Aspects of multivariate statistical theory}.
\newblock John-Wiley, 1982.

\bibitem[Nesterov and Nemirovski(1994)]{nestNem94}
Yu. Nesterov and A.~Nemirovski.
\newblock \emph{Interior-point polynomial algorithms in convex programming}.
\newblock SIAM, 1994.

\bibitem[Niculescu(2000)]{nicu00}
C.~P. Niculescu.
\newblock Convexity according to the geometric mean.
\newblock \emph{Mathematical Inequalities and Applications}, 3\penalty0
  (2):\penalty0 155--167, 2000.

\bibitem[Niculesu and Persson(2006)]{niculescu}
C.~Niculesu and L.~E. Persson.
\newblock \emph{Convex functions and their applications: a contemporary
  approach}, volume~13 of \emph{Science \& Business}.
\newblock Springer, 2006.

\bibitem[Nielsen and Bhatia(2013)]{nieBha13}
F.~Nielsen and R.~Bhatia, editors.
\newblock \emph{Matrix Information Geometry}.
\newblock Springer, 2013.

\bibitem[Ollila et~al.(2011)Ollila, Tyler, Koivunen, and Poor]{ollila11}
E.~Ollila, D.E. Tyler, V.~Koivunen, and H.~V. Poor.
\newblock Complex elliptically symmetric distributions: Survey, new results and
  applications.
\newblock \emph{{IEEE} Transactions on Signal Processing}, 60\penalty0
  (11):\penalty0 5597--5625, 2011.

\bibitem[Papadopoulos(2005)]{athan}
A.~Papadopoulos.
\newblock \emph{Metric spaces, convexity and nonpositive curvature}.
\newblock European Mathematical Society, 2005.

\bibitem[Qi et~al.(2010)Qi, Gallivan, and Absil]{qi10}
C.~Qi, K.~A. Gallivan, and P.-A. Absil.
\newblock Riemannian {BFGS} algorithm with applications.
\newblock In \emph{Recent Advances in Optimization and its Applications in
  Engineering}, pages 183--192. Springer Berlin Heidelberg, January 2010.

\bibitem[Rapcs\'ak(1991)]{rapsak}
T.~Rapcs\'ak.
\newblock Geodesic convexity in nonlinear optimization.
\newblock \emph{Journal of Optimization Theory and Applications}, 69\penalty0
  (1):\penalty0 169--183, 1991.

\bibitem[Ring and Wirth(2012)]{ring12}
W.~Ring and B.~Wirth.
\newblock Optimization methods on {Riemannian} manifolds and their application
  to shape space.
\newblock \emph{{SIAM} Journal on Optimization}, 22\penalty0 (2):\penalty0
  596--627, January 2012.

\bibitem[Rockafellar and Wets(1998)]{rockwets}
R.~T. Rockafellar and R.~J.-B. Wets.
\newblock \emph{Variational analysis}.
\newblock Springer, 1998.

\bibitem[Sra(2013)]{ssdiv}
S.~Sra.
\newblock Positive definite matrices and the symmetric stein divergence.
\newblock \emph{arXiv preprint arXiv: 1110.1773v4}, 2013.

\bibitem[Wiesel(2012)]{wie12}
A.~Wiesel.
\newblock Geodesic convexity and covariance estimation.
\newblock \emph{{IEEE} Transactions on Signal Processing}, 60\penalty0
  (12):\penalty0 6182--89, 2012.

\bibitem[Yuille and Rangarajan(2003)]{cccp}
A.~L. Yuille and A.~Rangarajan.
\newblock The concave-convex procedure.
\newblock \emph{Neural Computation}, 15\penalty0 (4):\penalty0 915--936, 2003.

\bibitem[Zhang et~al.(2013)Zhang, Wiesel, and Greco]{zhang13}
T.~Zhang, A.~Wiesel, and S.~Greco.
\newblock Multivariate generalized gaussian distribution: Convexity and
  graphical models.
\newblock \emph{{IEEE} Transaction on Signal Processing}, 60\penalty0
  (11):\penalty0 5597--5625, November 2013.

\bibitem[Zhu et~al.(2007)Zhu, Zhang, Ibrahim, and Peterson]{dti}
H.~Zhu, H.~Zhang, J.G. Ibrahim, and B.S. Peterson.
\newblock Statistical analysis of diffusion tensors in diffusion-weighted
  magnetic resonance imaging data.
\newblock \emph{Journal of the American Statistical Association}, 102\penalty0
  (480):\penalty0 1085--1102, 2007.

\end{thebibliography}

\end{document}